\documentclass{amsart}
\usepackage{hyperref}
\hypersetup{backref,
colorlinks=false}
\usepackage{amssymb}
\usepackage{amsmath,amsthm,bbm}
\usepackage{graphicx}
\usepackage{tikz}
\usepackage[margin=1.3in]{geometry}
\usepackage{lipsum} 
\usepackage{xcolor}

\def\R{{\mathbb R}}
\def\N{{\mathbb N}}

\def\1{{1\!\!1}}
\def\E{{\mathbb E}}
\def\P{{\mathbb P}}

\def\cal{\mathcal}
\def\eps{\varepsilon}

\newtheorem{theorem}{Theorem}
\newtheorem{proposition}{Proposition}
\newtheorem{lemma}{Lemma}

\newtheorem{remark}{Remark}

\title{Mean-field analysis of stochastic networks with reservation}

\author{Christine Fricker}
\address{(C. Fricker) DI-ENS, CNRS, INRIA de Paris, Paris, France}
\email{christine.fricker@inria.fr}

\author{Hanene Mohamed}
\address{(H. Mohamed) MODAL'X, UMR CNRS 9023, UPL, Universit\'e Paris Nanterre }
\email{hanene.mohamed@parisnanterre.fr}

\begin{document}


%

\begin{abstract} The problem of reservation in a large distributed system is analyzed via a new mathematical model. The target application is  car-sharing systems. This model is precisely motivated by the large station-based  car-sharing system in France, called Autolib'. This system can be described as a closed  stochastic   network where the nodes are the stations and the customers are the cars. The user can reserve the car  and the parking space. 
  In the paper, we study the evolution of the system when  the reservation of parking spaces and cars  is effective for all users.
  The asymptotic behavior of the underlying stochastic  network is given when the number $N$ of stations and the fleet size $M$ increase at the same rate. The analysis involves  a Markov process on a state space with dimension of order $N^2$. It is quite remarkable that the state process describing the evolution of the stations, whose dimension is of order $N$, converges in distribution, although not  Markov, to an non-homogeneous Markov process. We prove this  mean-field convergence. We also prove, using combinatorial arguments, that the mean-field limit has a unique equilibrium measure when the time between reserving and picking up the car is sufficiently small. This result extends the case  where only the parking space can be reserved.
\end{abstract}
 \maketitle

\tableofcontents

\section{Introduction}
The paper deals with a problem of reservation in a large distributed system. Our motivation is a car-sharing system in which a fleet of cars move around and  is parked in a set of stations, mainly for electric issues
. A crucial problem is the presence of empty and full stations. In an empty station,  users cannot pick up a car, while in a full station,  also called saturated,  users cannot park the car. 
Reservation could help the user to find both a car at the departure and a parking space at the destination.
In this paper, we focus  on a  reservation policy called {\em double reservation} which is to reserve both the car and the parking space at the same time, a moment before picking up the car.  Cars and  parking spaces reservation are proposed by Autolib', the car-sharing system that existed in Paris from $2011$ to $2018$. Its fleet was composed, in July $2016$, of $3\,980$ electric vehicles, called {\em  Bluecars}, distributed in $1\;084$ stations in Paris metropolitan area with $5\,935$ charging points. More than $126\,900$ subscribers had been registered for the service. See \cite{AutolibWiki} for more details about Autolib'.



Note that if the time between the reservation and the  pick-up  is zero, the policy is to reserve only the parking space when  the car is picked up.   This policy is called {\em simple reservation}. In free-floating systems
, users cannot reserve the parking spaces as the cars are parked in the public space. These systems are outside the scope of the paper.

\subsection{Simple reservation model}
The simple reservation model can be described  as follows. It consists of  $N$ stations of finite capacity $K$ and $M_N$ cars. 
Users arrive at rate $\lambda$ in each station. A user  reserves a parking space at  a randomly chosen   destination station  when he or she picks up a car.  If it is not possible, the unhappy user leaves the system. Otherwise, after a trip with  exponential distribution of parameter $\mu$, the car is parked at its  reserved parking space in the destination station.


The system is said to be large when $N$ and $M_N$ tend to infinity    at the same order. Only the fleet size $M_N$ is of the order of $N$. The other variables such as the arrival rate per station $\lambda$ or the mean trip duration $ 1/\mu$ are bounded, so of the order of $1$.  Indeed, the increase in the number of $N$ stations can be considered as a densification of the service area, and not as an extension. There is no reason for the mean trip time to increase. The aim of studying the model when $N$ tends to infinity is to obtain an approximation of a car-sharing system with a large fleet and a large number of stations. The aim is not to study the physical extension of the service area or the number of stations. 

Let us denote by $s_N$ the ratio $M_N/N$, tending to a constant $s$ which is  {\em the average number of cars per station}. This sizing parameter $s$
is a key parameter of the system. In \cite{bourdais2020mean}, this policy is studied  in a homogeneous framework by a mean-field  approach, using a large-scale  analysis similar to that of bike-sharing systems in \cite{fricker2016incentives}. 


Due to the   the parking space reservation, the state of each station is described in \cite{bourdais2020mean} by a vector with two components: the number of cars and the number of reserved parking spaces, which is the main difference with \cite{fricker2016incentives}. The state process is $2N$-dimensional and Markovian. With standard arguments, the mean-field convergence is established. Beyond this, the analysis of the equilibrium point is much more delicate with parking space reservation. 

The aim of the paper is to prove these results extended to the double reservation model introduced as follows. In particular, our papers gives the proof of  \cite[Theorem 2]{bourdais2020mean} omitted in~\cite{bourdais2020mean} on the simple reservation model.


\subsection{Double reservation model}
Station-based car-sharing systems  such Autolib' offer the possibility of reserving a car and a parking space in the desired station  online, a moment
before actually picking up the car. In this paper, we consider a model that meets this user demand. In our model, the user  reserves a car at a given  departure station and, at the same time, a parking space at a destination station. If there is no  car or no  parking space available, the user leaves the system. Otherwise, it takes a time called {\em reservation time} between the car reservation and the car  pick-up. The reservation time has an exponential distribution with mean $1/\nu$. Then follows the  trip whose duration is assumed to have an exponential distribution  with mean $1/\mu$. Eventually the user parks the car at the reserved space in the destination station and leaves the system.

\subsection{Discussion of the model}

\noindent

{\bf{Approximation of a large system}}:
      In practice the numbers of stations and cars are finite but large. For example, on 3 July 2016,  Autolib' offered $1\, 084$  stations for $3\, 980$ electric  cars. See~\cite{AutolibWiki}. The asymptotic analysis provides an approximation of the behavior of this finite-size system.  Therefore, the dimensioning problem of finding the optimal  number of cars per station for $N$ large is  approximated by its limit value when   $N$ tends to infinity.

{\bf{Double Vs simple reservation}}: 
For an order of magnitude of the variables of our model, note that Autolib' offered $30$ minutes as the maximum reservation duration for the car and $1$ hour $30$ minutes for the maximum reservation duration of the parking space at the destination (see~\cite[pp. 20]{waserhole:tel-01176190}). Moreover, the mean trip duration was around $38$ minutes (see  \cite{RapportAutolib2013} for $2013$). Thus, the mean trip time $1/\mu$ and the mean reservation time $1/\nu$ are comparable. This fully justifies the motivation of studying the double reservation policy.

\noindent

{\bf{Extension to an heterogeneous model}}: 
For sake of simplicity, our choice of a homogeneous network is motivated by the mean-field approach used in our study. A homogeneous framework is generally presented as the simplest, allowing the difficulties of the model to be highlighted and simple explicit results to be obtained. It is still possible to extend this study to a heterogeneous model using clusters by grouping stations with similar parameters in the real system. Since the trip times are random with an exponential distribution of parameter $\mu$, the heterogeneity of the trips is carried by the randomness. Finally, network heterogeneity is carried by the arrival rate depending on the cluster and the probability to reserve a parking space in a station of a given cluster (for the homogeneous model, this probability is $1/N$). See \cite{fricker2012mean} for details. For models with state process having a product-form invariant measure, the heterogeneous framework is quite natural. In the context of bike and car-sharing systems, see \cite{fricker2017equivalence, fricker2021stochastic}. However, our model does not fit  into this framework. For such models, the mean-field approach remains effective, hence our choice of a homogeneous model.

\noindent

{\bf{Extension to  heterogeneous users}}:  
To take account of different reservation behaviors of  users, two classes of users can be considered: users who make  reserve early  and users who reserve at the  last minute. This yields to introduce two different parameters $\nu_1$ and $\nu_2$ for the exponential distribution of the reservation time. This extension remains within a  Markovian framework.

Furthermore, the real-world trip time distribution can be fitted by an Erlang distribution. 
Replacing the exponential distribution of the trip time by an Erlang distribution, i.e. the distribution of the sum of i.i.d. random variables with exponential distribution, a state process capturing the phases of the  Erlang distribution is still Markovian.

\subsection{Main results and contribution}
\subsubsection{Mean-field approach}
A three-dimensional state space (reserved cars, reserved spaces and available cars) for each station does not allow this process to be fully described as a Markov process, because the parking space reservation time is the sum of two exponentially distributed variables (car pick-up time and  trip time). This leads to  the introduction of a fourth variable to distinguish between parking spaces reserved by {\em users not yet travelling} and {\em users travelling}. But even the associated state process is not Markov. As we will see on  the transitions described in Section~\ref{SectionModel}, the Markov process associated to this model, denoted by $(X^N(t))$, is more complicated  and especially of dimension $(N+1)^2$, where $N$ is the total number of stations. Therefore, for a given station, the state space is of dimension $N$, which is not suitable for an asymptotic  analysis when  $N$ tends to infinity. To solve this problem, we  introduce the process, denoted by $(Z^N(t))$, which describes the state of each station $i$, $1 \leq i \leq N$, as a function of the Markov process $(X^N (t))$.   Indeed $Z^N_i(t)$ has four components,
\begin{itemize}
\item[-] $R_i^{r,N}(t)$ the number of  parking places reserved by  non-driving users at station $i$ at time $t$,
\item[-] $R_i^{N}(t)$ the number of  parking places reserved by users driving at station $i$ at time $t$,
\item[-]  $V_i^{N}(t)$  the number of  cars available at station $i$ at time $t$ and
\item[-]  $V_i^{r,N}(t)$  the number of reserved cars at station $i$ at time $t$.
\end{itemize}

As  previously mentioned,  process $(Z^N(t))= \left(R^{r,N}(t), R^{N}(t), V^{N}(t), V^{r,N}(t)\right)$ is of dimension $4N$ but is not a Markov process. The loss of the Markov property is the price to pay to this   dimension reduction. Nevertheless, it is remarkable and  we prove that, despite this non-Markovian description, the state $(Z^N_i(t))$ of a given station $i$  ($1\leq i \leq N$)  converges in distribution, as N goes to infinity, to a non-homogeneous Markov process $(\overline{Z}(t))=(\overline{R^r}(t),\overline{R}(t),\overline{V}(t),\overline{V^r}(t))$ satisfying the following Fokker--Planck equation

\begin{align}\label{FPequation}
   \frac{d}{dt} \E\left(f(\overline{Z}(t))\right) &=
   \lambda \P( \overline{V}(t)>0)\E\left( (f(\overline{Z}(t)+e_1)-f(\overline{Z}(t)) 1_{\{ \bar{S}(t)<K\}} \right) \nonumber\\
   &+\nu \E\left((f(\overline{Z}(t)+e_2-e_1)-f(\overline{Z}(t)))1_{\{\overline{R^r}(t)>0\}}\right)\nonumber\\
 & +\mu \E\left((f(\overline{Z}(t)+e_3-e_2)-f(\overline{Z}(t)))1_{\{\overline{R}(t)>0\}}\right)\nonumber\\
 & + \lambda \P(\overline{S}(t)<K)\E\left((f(\overline{Z}(t)+e_4-e_3)-f(\overline{Z}(t)))1_{\{\overline{V}(t)>0\}}\right)\\
&  +\nu \E\left((f(\overline{Z}(t)-e_4)-f(\overline{Z}(t)))1_{\{\overline{V^r}(t)>0\}}\right)\nonumber
\end{align}
with $e_1=(1,0,0,0)$,  $e_2=(0,1, 0,0)$, $e_3=(0, 0,1,0)$, $e_4=(0, 0,0,1)$, $f$ a function with finite support on $\N^4$ and 
$\overline{S}(t) = \overline{R^{r}}(t) + \overline{R}(t) + \overline{V}(t)) + \overline{V^{r}}(t)$
the limiting number of unavailable parking spaces at station $i$ at time $t$. The asymptotic process $(\overline{Z}(t))=((\overline{R^r}(t),\overline{R}(t),\overline{V}(t),\overline{V^r}(t))$ is a jump process with time-dependent rates, the so-called McKean--Vlasov process. This mean-field convergence theorem is not standard at all. Indeed, the mean-field limit is usually obtained for a Markov state process.

\subsubsection{The equilibrium} For non-homogeneous Markov processes, recall that there can be several invariant measures. We prove that there is a unique invariant measure in a restricted framework.
The proof is based on  three main arguments. First, by applying queueing theory, the McKean--Vlasov process on the basic state space is identified with a tandem of four queues with an  invariant measure of explicit product-form.  This explicit product-form allows us to change the problem of existence and uniqueness of the invariant measure of  the non-homogeneous Markov process to the same problem for a fixed point equation in dimension $4$.  Then, by  straightforward simplifications, the last fixed point equation is reduced to dimension $2$. This simplification is straightforwardly proved by the queueing interpretation of the limiting McKean--Vlasov process.   Finally, the  global inversion theorem and a monotonicity property allow us to conclude. The last two arguments are based on combinatorial calculations.  Monotonicity is just proved when the mean reservation time is sufficiently  small. We are convinced that  this assumption is  technical  but the proof needs other tools.  This monotonicity property is first guessed on the tandem of  queues. It is a main benefit of the queueing interpretation of the limiting process.


\subsection{Related works} 
 For sharing systems, despite the need of analysis of these stochastic systems, most of the literature concerns operation research and data analysis. Probabilistic models have been proposed and analyzed for bike-sharing systems \cite{fricker2016incentives, fricker2017equivalence}, where usage does not allow reservation. As far as we know, there is no stochastic analysis for car-sharing systems with reservation before \cite{bourdais2020mean}.

For car-sharing systems, a first part of the literature investigates the location problem. In Boyaci et al. \cite{boyaci2015optimization}, OR optimization is used to plan an efficient car-sharing system in terms of the number, location and capacity of stations and fleet size, applied to the case  of Nice, France.

Vehicle redistribution  and staff rebalancing is a big issue in car-sharing systems (see for example \cite{nourinejad2015vehicle}). For simple reservation, called complete parking reservation (CPR),  Kaspi et al. 
  show by simulation in \cite{kaspi2014parking} and then with OR techniques in \cite{kaspi2016regulating} that CPR outperforms NR (no reservation) for a specific user-oriented metric. This metric is the  excess time   users spend in the system due to lack of cars or parking spaces, i.e. the difference between  actual time and   ideal  trip time, including walking times to  stations.   Paper~\cite{deng2020fleet} does not  investigate the reservation impact, but the charging problem. A queueing analysis for dimensioning the fleet size is proposed for a closed network  taking account car charging. 

Despite the potential of car sharing, even data analysis  remains largely unexplored. This is also due to the lack of data provided by the operators. In \cite{boldrini2016characterising}, Boldrini et al.  exploit one month (April 2015) of publicly available data from Autolib', in Paris, France, with $960$ stations and  $2\,700$ electric cars at this time. We get an idea of the average car pickup rate and the availability of cars at a station. Furthermore, a dichotomy between Paris and the suburbs is highlighted.

Mean-field is an efficient tool for studying the behavior of large distributed systems or interacting systems  in many different domains fo applications. These systems have  both a large number of nodes and customers (particles, cars, etc).  
To the best of our knowledge, our large-scale stochastic analysis is the first  for   a stochastic model of car sharing systems with double reservation. Our first main result is that the non-Markovian state process for a given station converges to a non-homogeneous Markov process which is not standard. Furthermore, we note that this limiting Markov process  can be described using a simple queueing system. This result was first obtained by simulation in~\cite[Section 5]{bourdais2015stochastic}. Indeed, in \cite{bourdais2015stochastic},   an artificial   Markov model on the state of the stations, called the approximate model process, is introduced for the double reservation. Its mean-field limit at equilibrium fits with the real  dynamics at equilibrium, which allows the authors to guess such an outcome. The same phenomenon is proved in an entirely different framework, for a  model of a network with failures in \cite{aghajani2018large}. Note that the framework in \cite{aghajani2018large} induces simultaneous jumps, which make the proofs more technical. Our paper discusses a simpler framework which focuses on the main arguments and it provides the proof of the  result  expected in~\cite{bourdais2015stochastic}.   Like our model, the celebrated Gibbens-Hunt-Kelly model exhibits strong interactions (\cite{gibbens1990bistability}). In \cite{graham1993propagation}, it is proved that these interactions disappear  for the mean-field limit. Because of these three models, we believe that this methodology can be useful in many other contexts. A main contribution of the paper relies on the result of uniqueness of the equilibrium point in high dimension (dimension  4), thanks to a nice  interpretation in terms of queues. As far as we know, there is no such  result in the literature.

\subsection*{Outline of the paper}
  In Section  \ref{SectionModel},  the Markov process $(X^N(t))$ associated to our model is defined and 
 the stochastic evolution equations are given. In Section~\ref{SectionAsymptoticProcess},  the second process $(Z^N(t))$  describing the state of the stations is introduced. A heuristic computation of its McKean--Vlasov asymptotic process is derived. Theorem~\ref{McKean-Vlasov} which establishes the existence and uniqueness of  this stochastic process is proved.  Section~\ref{SectionMean-FieldLimit}  is devoted to  Theorem \ref{mean-field} giving the mean-field convergence for $(Z^N(t))$ and highlights the probabilistic interpretation of the Mckean--Vlasov process.  Section \ref{SectionEquilibrium}  analyses of its invariant distribution.


\section{The  model} \label{SectionModel}
In this section, we describe the dynamics of our stochastic model. We recall that the system has $N$ stations of capacity $K$. The Markov process that gives  the dynamics of the system is the following,
\begin{align*}
  X(t)=(X^N_{i,j}(t),0\leq i,j\leq N)
\end{align*}
where, for $1\leq i,j\leq N$,  at time $t$,

\begin{itemize}
\item[-] $X^N_{i,j}(t)$ is the number of cars  reserved at $i$ with  parking space reserved at $j$,

\item[-]  $X^N_{0,j}(t)$ is the number of   parking spaces reserved at $j$ by users driving

\item[-] and $X^N_{i,0}(t)$ is the number of   cars available at station $i$.
\end{itemize}

The total number of cars is $M_N$ and the fleet size parameter,  defined as the mean number of cars per station if they are all parked, is
\begin{align*}
  s \stackrel{def.}{=}\lim_{N\rightarrow +\infty} \frac{M_N}{N}.
\end{align*}

\subsection{Transitions of the Markov  process}
The transitions of the Markov  process $(X^N(t))$ are  reservations, car picks up and car returns. 
\begin{itemize}
\item[-] {\em Reservations}. At station $i$, at rate $\lambda$,   an available car is replaced by a reserved car and an available parking space  is reserved at the same time at a random station, say  $j$. If there is either no available car at $i$ or no available parking space at $j$, the reservation fails. If   a reservation  is made at time $t$,
\begin{align*}
\begin{cases}
X^N_{i,0}(t)=X^N_{i,0}(t^-)-1 & \text { for } 1\leq i\leq N,\\
X^N_{i,j}(t)=X^N_{i,j}(t^-)+1 & \text { for  } 1\leq i, j\leq N.\\
\end{cases}
\end{align*}
where the  limit from the left of function $f$ at $t$ is denoted by $f(t^-)$. 
  
\item[-] {\em Car picks up}. After a reservation, the user takes a time with an exponential distribution with parameter $\nu$ to come and pick the car. So, at rate $\nu$, each  car reserved at station $i$ disappears and the associated  parking space reserved at station $j$ moves to a  parking space reserved by a driving user. When  taking such a car  at time $t$,
\begin{align*}
\begin{cases}
  X^N_{i,j}(t)=X^N_{i,j}(t^-)-1, & \text { for  } 1\leq i,j\leq N,\\
  X^N_{0,j}(t)=X^N_{0,j}(t^-)+1 & \text { for  } 1\leq j\leq N,\\
\end{cases}
\end{align*}

\item[-] {\em Car returns}. After a trip with exponential distribution with parameter $\mu$, a user driving returns his car. Thus, at rate $\mu$, each parking space at station $j$ reserved by a  user  driving his car becomes an available car. When  a car is returned at station $j$ at time $t$,
\begin{align*}
\begin{cases}
 X^N_{0,j}(t)=X^N_{0,j}(t^-)-1,\\
X^N_{j,0}(t)=X^N_{j,0}(t^-)+1.
\end{cases}
\end{align*}
\end{itemize}
Note that $(X^N(t))$  is an irreducible Markov process on the finite state space
\begin{align}\label{statespace}
  \{x=(x_{i,j})\in \N^{(N+1)^2}, \sum_{j=0}^N x_{i,j}+ \sum_{j=0}^N x_{j,i}\leq K, \sum_{0\leq i,j\leq N}^N x_{i,j}=M_N\},
\end{align}
where $K$ is the finite capacity  of each station, thus $(X^N(t))$ is ergodic.

\subsection{Stochastic evolution equations} The dynamics of $(X^N(t))$ can be given in terms of stochastic integrals with respect to Poisson processes. Let us introduce the following notations.
\begin{itemize}
\item[-] A Poisson process on $\R_+$ with parameter $\xi$ is denoted by $\cal{N}_{\xi}$. A sequence of such i.i.d.  processes is denoted by $(\cal{N}_{\xi,i}, i\in \N)$.\\
\item[-] A Poisson process on $\R_+^2$ with intensity $\xi dt dh$ is denoted by $\overline{\cal{N}}_{\xi}$. A sequence of such i.i.d.  processes is denoted by $(\overline{\cal{N}}_{\xi,i}, i\in \N)$.\\
\item[-] A marked Poisson process $(t_n,U_n)$, where $(t_n,n\in \N)$ is a Poisson process on $\R_+$ with parameter $\xi$ and $(U_n, n\in \N)$  is a sequence of i.i.d random variables with uniform distribution on $\{1,\ldots,N\}$, is denoted by
$\cal{N}_{\xi}^{U,N}$. Note that, for $1\leq i \leq N$, $\cal{N}_{\xi}^{U,N}(.,\{i\})$
is a Poisson process on $\R_+$ with parameter $\xi/N$ and $\cal{N}_{\xi}^{U,N}(.,\N)$ is a Poisson process on $\R_+$ with parameter $\xi$.
\end{itemize}

Let us introduce the following   independent point processes.  A reservation of a car in station $i$  is a point of a Poisson process $\cal{N}_{\lambda,i}^{U,N}$ on $\R_+^2$.
  For  reservations of both a car  at station $i$ and a parking space  at  station $j$, the times from the moment the user makes a reservation to the moment the car is picked up  form a  Poisson process $\cal{N}_{\nu,i,j}$  on $\R_+$. We need a sequence of such i.i.d processes $(\cal{N}_{\nu,i,j,l},l\in \N)$  as cars are picked up independently. And the same for the trip times of cars returned at station $j$, associated to a sequence $(\cal{N}_{\mu,j,l},l\in \N)$ of independent Poisson processes.

Using the previous  notations, process $(X^N(t))$ is given by the following stochastic differential equations. For $1\leq i, j \leq N$ and $t\geq 0$,
\begin{align}\label{evolution1}
  d X^N_{i,j}(t)=  &-\sum_{l=1}^{\infty}\mathbf{1}_{\{l\leq X^N_{i,j}(t^-)\}}\cal{N}_{\nu,i,j,l}(dt) \\
  &+ \mathbf{1}_{\{X^N_{i,0}(t^-)>0,\; \sum_{k=0}^N \left(X^N_{k,j}(t^-)+X^N_{j,k}(t^-)\right)<K\}}\; \cal{N}_{\lambda,i}^{U,N}(dt,\{j\}),\nonumber
  \end{align}
  \begin{align}\label{evolution2}
  d X^N_{0,j}(t)= &\sum_{l=1}^{\infty}\sum_{i=1}^N \mathbf{1}_{\{l\leq X^N_{i,j}(t^-)\}}\cal{N}_{\nu,i,j,l}(dt)-\sum_{l=1}^{\infty}\mathbf{1}_{\{l\leq X^N_{0,j}(t^-)\}}\cal{N}_{\mu,j,l}(dt)
  \end{align}
  and
  \begin{align}\label{evolution3}
    d X^N_{i,0}(t)=&-\sum_{j=1}^N \mathbf{1}_{\{X^N_{i,0}(t^-)>0,\; \sum_{k=0}^N \left(X^N_{k,j}(t^-)+X^N_{j,k}(t^-)\right)<K\}}\; \cal{N}_{\lambda,i}^{U,N}(dt,\{j\})\\
   &+ \sum_{l=1}^{\infty}\mathbf{1}_{\{l\leq X^N_{0,i}(t^-)\}}\cal{N}_{\mu,i,l}(dt). \nonumber
  \end{align}

  \section{The asymptotic process} \label{SectionAsymptoticProcess}
\subsection{Introduction to the asymptotic process}
Some other notations are needed for that.
  The state of each station $i$ is given by the quadruplet 
  \begin{align*}
    Z_i^N(t)=(R^{r,N}_i(t),R^{N}_i(t),V^{N}_i(t),V^{r,N}_i(t))
  \end{align*}
  where, at node $i$ at time $t$,
  \begin{itemize}
  \item[-] $R^{r,N}_i(t)$ is the number of   parking spaces reserved by  non-driving users,
    
  \item[-]  $R^N_{i}(t)$ is the number of  reserved parking spaces by users driving,
  
  \item[-] $V^N_i(t)$ is the number of available cars,
  
  \item[-] $V^{r,N}_i(t)$  is the number of reserved cars.
  \end{itemize}
  \noindent
  Process $(Z^N(t))$ takes values on
  \begin{align*}
    \cal{S}^N=\{(w_i,x_i,y_i,z_i)_{1\leq i\leq N}, (w_i,x_i,y_i,z_i)\in \Sigma_K, \sum_{i=1}^N w_i+x_i+y_i+z_i=M_N  \}
  \end{align*}
  where
  \begin{align*}
    \Sigma_K=\{(w,x,y,z) \in \N^4, w+x+y+z\leq K \}.
  \end{align*}
\noindent
 Moreover, let us denote by
 \begin{align*}
   S^N_i(t)=R^{r,N}_i(t)+R^{N}_i(t)+V^{N}_i(t)+V^{r,N}_i(t)
 \end{align*}
 the number of unavailable parking spaces at station $i$ at time $t$. 
Note that $ S^N_i(t)$  is a function of $Z^N_i(t)$.
  This process $(Z^N(t))$ gives some refined state of the stations and can be obtained as a function of the Markov process $(X^N(t))$ by
 \begin{align*}
   R^{r,N}_i(t)&= \sum_{j=1}^N X^N_{j,i}(t), &  R^{N}_i(t)=  X^N_{0,i}(t),\\
   V^{N}_i(t)&=  X^N_{i,0}(t), & V^{r,N}_i(t)= \sum_{j=1}^N X^N_{i,j}(t).
 \end{align*}

 So the evolution equations of $(Z^N(t))$ can be obtained from equations~\eqref{evolution1}--\eqref{evolution3} as follows
 \begin{align}
   d R^{r,N}_i(t)= &\sum_{j=1}^N \mathbf{1}_{\{V^N_{j}(t^-)>0,\; S^N_{i}(t^-)<K\}}\; \cal{N}_{\lambda,j}^{U,N}(dt,\{i\}) \label{state1}\\
   &-\sum_{j=1}^N \left( \sum_{l=1}^{\infty}\mathbf{1}_{\{l\leq X^N_{j,i}(t^-)\}}\cal{N}_{\nu,j,i,l}(dt)\right) , \nonumber\\
  d R^{N}_i(t)=  &\sum_{l=1}^{\infty}\sum_{j=1}^N \mathbf{1}_{\{l\leq X^N_{j,i}(t^-)\}}\cal{N}_{\nu,j,i,l}(dt)-\sum_{l=1}^{\infty}\mathbf{1}_{\{l\leq R^N_i(t^-)\}}\cal{N}_{\mu,i,l}(dt),\\
   d V^{N}_i(t)=& \sum_{l=1}^{\infty}\mathbf{1}_{\{l\leq R^N_i(t^-)\}}\cal{N}_{\mu,i,l}(dt)\\
   &-\sum_{j=1}^N \mathbf{1}_{\{V^N_i(t^-)>0,\; S^N_j(t^-)<K\}}\; \cal{N}_{\lambda,i}^{U,N}(dt,\{j\}), \nonumber\\  
 d V^{r,N}_i(t)= & \sum_{j=1}^N \mathbf{1}_{\{V^N_i(t^-)>0,\; S^N_j(t^-)<K\}}\; \cal{N}_{\lambda,i}^{U,N}(dt,\{j\})\label{state4}\\
  &-\sum_{j=1}^N \left( \sum_{l=1}^{\infty}\mathbf{1}_{\{l\leq X^N_{i,j}(t^-)\}}\cal{N}_{\nu,i,j,l}(dt)\right) .\nonumber
\end{align}

This process $(Z^N(t))$, in dimension $4N$, is not a Markov process because the evolution equations for $(Z^N(t))$ are not autonomous. They depend on the Markov process $(X^N(t))$ which lives in dimension $(N+1)^2$. See equation~\eqref{statespace}. 
We introduce process $(Z^N(t))$ because it is sufficient to capture the performance of the system, such as the fact that a station is empty or full. The quite remarkable property is that the limit  as $N$ gets large of $(Z^N(t))$ is a non-linear Markov process. We will present this process in this section and   prove the convergence result in the next section.

\subsection{Heuristic computation of the asymptotic process}
The  proof of the   convergence will be given in the next section. Here we show how we can guess the asymptotic process.
Suppose that, for $1\leq i\leq N$, $(Z^N_i(t))$ converges in distribution to some process
\begin{align*}
 (\overline{Z}(t))=(\overline{R^r}(t),\overline{R}(t),\overline{V}(t),\overline{V^r}(t)).
\end{align*}
 Let $\cal{P}_i^N$ be the random measure  on $\R_+$ defined by
\begin{align}\label{defP}
  \cal{P}_i^N([0,t])= \int_0^t \sum_{j=1}^N \left( \sum_{l=1}^{\infty}\mathbf{1}_{\{l\leq X^N_{j,i}(s^-)\}}\cal{N}_{\nu,j,i,l}(ds)\right).
\end{align}

It shows that $\cal{P}_i^N$ is a counting process (with jump size $1$) on $\R_+$ and its compensator is
\begin{align*}
  \nu  \int_0^t \sum_{j=1}^N  X^N_{j,i}(s)\; ds=  \nu  \int_0^t R^{r,N}_i(s)\; ds.
\end{align*}
See Robert \cite[Proposition A.9]{robert2013stochastic} for example. Thus, due to the convergence in distribution of $(Z^N_i(t))$ and the standard results on convergence of point processes,  $\cal{P}_i^N$ converges to an non-homogeneous Poisson process $\cal{P}^{\infty}$ with intensity $(\nu \overline{R^r}(t))$. It can be written as follows
\begin{align*}
  \cal{P}^{\infty}(dt)= \int_{\R_+} \mathbf{1}_{\{0\leq h\leq \overline{R^r}(t^-)\}}\overline{\cal{N}}_{\nu,1}(dt,dh)
\end{align*}
where, with our notations, $\overline{\cal{N}}_{\nu,1}$ is a Poisson process with intensity $\nu dh dt$,
using  the characterization of a Poisson process by the  martingale of its stochastic integral.

Following the same lines, it gives also that $\tilde{\cal{P}}_i^N$,  the random measure  on $\R_+$ defined by
\begin{align}\label{defPtilde}
 \tilde{ \cal{P}}_i^N([0,t])= \int_0^t \sum_{j=1}^N \left( \sum_{l=1}^{\infty}\mathbf{1}_{\{l\leq X^N_{i,j}(s^-)\}}\cal{N}_{\nu,i,j,l}(ds)\right)
\end{align}
is a counting process (with jump size $1$) on $\R_+$ with compensator 
\begin{align*}
  \nu  \int_0^t \sum_{j=1}^N  X^N_{i,j}(s)\; ds=  \nu  \int_0^t V^{r,N}_i(s)\; ds.
\end{align*}
It converges to an non-homogeneous Poisson process $\tilde{\cal{P}}^{\infty}$ with intensity $(\nu \overline{V^r}(t))$, i.e.
\begin{align*}
  \tilde{\cal{P}}^{\infty}(dt)= \int_{\R_+} \mathbf{1}_{\{0\leq h\leq \overline{V^r}(t^-)\}}\overline{\cal{N}}_{\nu,2}(dt,dh).
\end{align*}

\begin{remark}\label{rem} Point processes $\overline{\cal{N}}_{\nu,1}$ and $\overline{\cal{N}}_{\nu,2}$ are independent because   $\cal{P}_i^N$ (respectively $\tilde{ \cal{P}}_i^N$) is a function of  $(\cal{N}_{\nu,i,j,.},j)$ (respectively $(\cal{N}_{\nu,j,i,.},j)$), where $(\cal{N}_{\nu,i,j,.},j\not =i)$ and $(\cal{N}_{\nu,j,i,.},j\not =i)$ are independent. 
\end{remark}

Then let us consider  the random  measure $\cal{Q}_i^N$ on $\R_+$ defined by
\begin{align}\label{defQ}
    \cal{Q}_i^N([0,t])= \int_0^t \sum_{j=1}^N \mathbf{1}_{\{V^N_{j}(s^-)>0,\; S^N_{i}(s^-)<K\}}\; \cal{N}_{\lambda,j}^{U,N}(ds,\{i\}) 
\end{align}
with compensator
\begin{align*}
  \lambda  \int_0^t \mathbf{1}_{ \{S^N_{i}(s)<K\}} \frac{1}{N} \sum_{j=1}^N \mathbf{1}_{\{V^N_{j}(s)>0\}}  ds.
\end{align*}
Heuristically, by asymptotic independence of stations  and the law of large numbers, as $N$ tends to infinity, 
\[
\frac{1}{N} \sum_{j=1}^N \mathbf{1}_{\{V^N_{j}(t)>0\}}\rightarrow \P(\overline{V}(t)>0).
\]
Thus  formally, as $N$ tends to $+\infty$, $\cal{Q}_i^N$ converges to an non-homogeneous Poisson process  $\cal{Q}^{\infty}$ with intensity $(\lambda \mathbf{1}_{ \{\overline{S}(t)<K\}}\P(\overline{V}(t)>0))$.
In other words, 
\begin{align*}
  \cal{Q}^{\infty}(dt)=\int_{\R_+} \mathbf{1}_{\{0\leq h\leq \mathbf{1}_{\overline{S}(t^-)<K\}}\P(\overline{V}(t^-)>0)\}}\overline{\cal{N}}_{\lambda,1}(dt,dh).
\end{align*}
With exactly the same work, 
$\tilde{\cal{Q}}_i^N$ on $\R_+$ defined by
\begin{align}\label{defQtilde}
   \tilde{ \cal{Q}}_i^N([0,t])= \int_0^t \sum_{j=1}^N \mathbf{1}_{\{V^N_{i}(s^-)>0,\; S^N_{j}(s^-)<K\}}\; \cal{N}_{\lambda,i}^{U,N}(ds,\{j\}) 
\end{align}
formally converges in distribution when  $N$ tends to $+\infty$ to the non-homogeneous  Poisson process $\tilde{\cal{Q}}^{\infty}$  with intensity $(\lambda \mathbf{1}_{ \{\overline{V}(t)>0\}}\P(\overline{S}(t)<K))$ also defined by
\begin{align*}
  \tilde{ \cal{Q}}^{\infty}(dt)=\int_{\R_+} \mathbf{1}_{\{0\leq h\leq \mathbf{1}_{\{\overline{V}(t^-)>0\}}\P(\overline{S}(t^-)<K)\}}\overline{\cal{N}}_{\lambda,2}(dt,dh).
\end{align*}
Note that for the same reason as previously, $\overline{\cal{N}}_{\lambda,2}$ is independent of $\overline{\cal{N}}_{\lambda,1}$.
By  formally taking the limit in equations~\eqref{state1}--\eqref{state4},
it leads to
\begin{align}\label{limit}
  \begin{cases}
    d \overline{R^r}(t)=&\int_{\R_+} \mathbf{1}_{\{0\leq h\leq \mathbf{1}_{\{\overline{S}(t^-)<K\}}\P(\overline{V}(t^-)>0)\}}\overline{\cal{N}}_{\lambda,1}(dt,dh)
    -\int_{\R_+} \mathbf{1}_{\{0\leq h\leq \overline{R^r}(t^-)\}}\overline{\cal{N}}_{\nu,1}(dt,dh),\\
    \\
    d\overline{R}(t)= &\int_{\R_+} \mathbf{1}_{\{0\leq h\leq \overline{R^r}(t^-)\}}\overline{\cal{N}}_{\nu,1}(dt,dh) -\sum_{l=1}^{+\infty} \mathbf{1}_{\{l\leq \overline{R}(t^-)\}}\cal{N}_{\mu,l}(dt),\\
    \\
    d \overline{V}(t)= &\sum_{l=1}^{+\infty} \mathbf{1}_{\{l\leq \overline{R}(t^-)\}}\cal{N}_{\mu,l}(dt)
    -\int_{\R_+} \mathbf{1}_{\{0\leq h\leq \mathbf{1}_{\{\overline{V}(t^-)>0\}}\P(\overline{S}(t^-)<K)\}}\overline{\cal{N}}_{\lambda,2}(dt,dh),\\
    \\
    d \overline{V^r}(t)=&\int_{\R_+} \mathbf{1}_{\{0\leq h\leq \mathbf{1}_{\{ \overline{V}(t^-)>0\}}\P(\overline{S}(t^-)<K)\}}\overline{\cal{N}}_{\lambda,2}(dt,dh)-\int_{\R_+} \mathbf{1}_{\{0\leq h\leq \overline{V^r}(t^-)\}}\overline{\cal{N}}_{\nu,2}(dt,dh).
   \end{cases}
\end{align}

\subsection{A first result} Then the first result gives the existence and uniqueness of a stochastic process solution of the system of  SDEs~\eqref{limit}. Let $T>0$ be fixed. Let $\cal{D}_T=\cal{D}([0,T],\cal{P}(\Sigma_K))$ be the set of cad-lag functions from $[0,T]$ to $\cal{P}(\Sigma_K)$.

\begin{theorem}[McKean--Vlasov process]\label{McKean-Vlasov}
   For every $(w,x,y,z)\in \Sigma_K$, the system of equations
      \begin{align}\label{uniquelimit}
  \begin{cases}
    \overline{R^r}(t)=&w+\iint_{[0,t]\times \R_+} \mathbf{1}_{\{0\leq h\leq \mathbf{1}_{\{\overline{S}(s^-)<K\}}\P(\overline{V}(s^-)>0)\}}\overline{\cal{N}}_{\lambda,1}(ds,dh)\\
    \\
    & -\iint_{[0,t]\times \R_+} \mathbf{1}_{\{0\leq h\leq \overline{R^r}(s^-)\}}\overline{\cal{N}}_{\nu,1}(ds,dh),\\
    \\
    \overline{R}(t)= & x+ \iint_{[0,t]\times \R_+} \mathbf{1}_{\{0\leq h\leq \overline{R^r}(s^-)\}}\overline{\cal{N}}_{\nu,1}(ds,dh)\\
    \\& -\int_0^t \sum_{l=1}^{+\infty} \mathbf{1}_{\{l\leq \overline{R}(s^-)\}}\cal{N}_{\mu,l}(ds),\\
    \\
    \overline{V}(t)= & y+\int_0^t \sum_{l=1}^{+\infty} \mathbf{1}_{\{l\leq \overline{R}(s^-)\}}\cal{N}_{\mu,l}(ds)\\
    \\
    & -\iint_{[0,t]\times \R_+} \mathbf{1}_{\{0\leq h\leq \mathbf{1}_{\{\overline{V}(s^-)>0\}}\P(\overline{S}(s^-)<K)\}}\overline{\cal{N}}_{\lambda,2 }(ds,dh)\\
   \\
    \overline{V^r}(t)=& z+ \iint_{[0,t]\times \R_+} \mathbf{1}_{\{0\leq h\leq \mathbf{1}_{\{\overline{V}(s^-)>0\}}\P(\overline{S}(s^-)<K)\}}\overline{\cal{N}}_{\lambda,2}(ds,dh)\\
    \\
    & -\iint_{[0,t]\times \R_+} \mathbf{1}_{\{0\leq h\leq \overline{V^r}(s^-)\}}\overline{\cal{N}}_{\nu,2}(ds,dh),
   \end{cases}
   \end{align}   
    has a unique solution $(\overline{R^r}(t),\overline{R}(t),\overline{V}(t),\overline{V^r}(t))$ in $\cal{D}_T$.
  \end{theorem}

  Note that the solution of equation~\eqref{uniquelimit} satisfies the Fokker--Planck equation~\eqref{FPequation} in the introduction.
  \begin{proof}
     The proof is standard and quite technical. Let us introduce the Wasserstein distances on $\cal{P}(\cal{D}_T)$. For $\pi_1$, $\pi_2\in \cal{P}(\cal{D}_T)$,
    \begin{align*}
      W_T(\pi_1, \pi_2)&=\inf_{\pi \in \cal{C}_T(\pi_1, \pi_2)} \int_{\omega=(\omega_1,\omega_2)\in \cal{D}_T^2} \left(d_T(\omega_1,\omega_2)\wedge 1\right) d\pi(\omega),\\
     \rho _T(\pi_1, \pi_2)&=\inf_{\pi \in \cal{C}_T(\pi_1, \pi_2)} \int_{\omega=(\omega_1,\omega_2)\in \cal{D}_T^2} \left(\| \omega_1-\omega_2\|_{\infty,T})\wedge 1\right) d\pi(\omega)
    \end{align*}
    where, for $f\in \cal{D}_T$, $\|f\|_{\infty,T}=\sup\{\|f\|,0\leq t\leq T\}=\sup\{\sum_{i=1}^4 |f_i(t)|,0\leq t\leq T\}$, $\cal{C}_T(\pi_1, \pi_2)$ is the set of couplings of $\pi_1$ and  $\pi_2$, i.e. the subset of $\cal{P}(\cal{D}_T^2)$ with first marginal is $\pi_1$ and the second is $\pi_2$, $(\cal{D}_T,d_T)$ with $d_T$ the distance associated to the Skorohod topology is complete and separable thus $(\cal{P}(\cal{D}_T),W_T)$ is complete and separable and $W_T(\pi_1, \pi_2)\leq \rho _T(\pi_1, \pi_2)$.
    
    Let us define
    \begin{align*}
      \Phi:  (\cal{P}(\cal{D}_T),W_T) & \rightarrow (\cal{P}(\cal{D}_T),W_T)\\
      \pi & \mapsto \Phi(\pi)
    \end{align*}
    where $\Phi(\pi)$ is the distribution of $(Z_{\pi}(t))=(R^r_{\pi}(t),R_{\pi}(t),V_{\pi}(t),V^r_{\pi}(t))$, unique solution of the SDEs
    \vspace{-1cm}
    
        \begin{align}\label{Z_pi}
  \begin{cases}
    R^r_{\pi}(t)=&w+\iint_{[0,t]\times \R_+} \mathbf{1}_{\{0\leq h\leq \mathbf{1}_{\{\|Z_{\pi}(s^-)\|<K\}}\pi(v(s^-)>0)\}}\overline{\cal{N}}_{\lambda,1}(ds,dh)\\
    \\
    & -\iint_{[0,t]\times \R_+} \mathbf{1}_{\{0\leq h\leq R^r_{\pi}(s^-)\}}\overline{\cal{N}}_{\nu,1}(ds,dh),\\
    \\
    R_{\pi}(t)= & x+ \iint_{[0,t]\times \R_+} \mathbf{1}_{\{0\leq h\leq R^r_{\pi}(s^-)\}}\overline{\cal{N}}_{\nu,1}(ds,dh)\\
    \\& -\int_0^t \sum_{l=1}^{+\infty} \mathbf{1}_{\{l\leq R_{\pi}(s^-)\}}\cal{N}_{\mu,l}(ds),\\
    \\
    V_{\pi}(t)= & y+\int_0^t \sum_{l=1}^{+\infty} \mathbf{1}_{\{l\leq R_{\pi}(s^-)\}}\cal{N}_{\mu,l}(ds)\\
    \\
    & -\iint_{[0,t]\times \R_+} \mathbf{1}_{\{0\leq h\leq \mathbf{1}_{\{V_{\pi}(s^-)>0\}}\pi(\|z(s^-)\|<K)\}}\overline{\cal{N}}_{\lambda,2 }(ds,dh)\\
   \\
    V^r_{\pi}(t)=& z+ \iint_{[0,t]\times \R_+} \mathbf{1}_{\{0\leq h\leq \mathbf{1}_{\{V_{\pi}(s^-)>0\}}\pi(\|z(s^-)\|<K)\}}\overline{\cal{N}}_{\lambda,2}(ds,dh)\\
    \\
    & -\iint_{[0,t]\times \R_+} \mathbf{1}_{\{0\leq h\leq V^r_{\pi}(s^-)\}}\overline{\cal{N}}_{\nu,2}(ds,dh).
   \end{cases}
        \end{align}
        Note that $\pi(v(t)>0)=\int_{z=(r^r,r,v,v^r)\in \cal{D}_T} 1_{\{v(t)>0\}} d\pi(z)$. The existence and uniqueness of a solution to~\eqref{uniquelimit} is equivalent to the existence and uniqueness of a fixed point $\pi=\Phi(\pi)$.

        Let us prove that $\pi=\Phi(\pi)$ has a unique fixed point.
For $\pi_1$, $\pi_2\in \cal{P}(\cal{D}_T)$, let $Z_{\pi_1}$ and $Z_{\pi_2}$ be  solutions of~\eqref{Z_pi}. 
        Thus  $(Z_{\pi_1},Z_{\pi_2})$ is a coupling of $\Phi(\pi_1)$ and $\Phi(\pi_2)$ and, for $ t\leq T$,
        $ \rho _t(\Phi(\pi_1), \Phi(\pi_2))\leq \E\left( \|Z_{\pi_1}-Z_{\pi_2}\|_{\infty,t}\right)$.
       
        For $t\leq T$, using the definition~\eqref{Z_pi} of $Z_{\pi_1}$ and $Z_{\pi_2}$, 
       \begin{align}\label{diffZ_pi}
        & \| Z_{\pi_1} -Z_{\pi_2} \|_{\infty,t}\nonumber\\
         & =\sup_{0\leq s\leq t} \left( |R^r_{\pi_1}(s)-R^r_{\pi_2}(s)|+|R_{\pi_1}(s)-R_{\pi_2}(s)|+|V_{\pi_1}(s)-V_{\pi_2}(s)|+|V^r_{\pi_1}(s)-V^r_{\pi_2}(s)| \right)\nonumber \\
         & \qquad  \leq \iint_{[0,t]\times \R_+} \mathbf{1}_{\{A_{\pi_1}(s^-) \wedge  A_{\pi_2}(s^-) \leq  h \leq   A_{\pi_1}(s^-) \vee  A_{\pi_2}(s^-)\}}\overline{\cal{N}}_{\lambda,1}(ds,dh)\nonumber\\
         & \qquad + 2 \int_0^t \sum_{l=1}^{+\infty} |\mathbf{1}_{\{l\leq R_{\pi_1}(s^-)\}}-\mathbf{1}_{\{l\leq R_{\pi_2}(s^-)\}}| \cal{N}_{\mu,l}(ds)\nonumber\\
         & \qquad +2  \iint_{[0,t]\times \R_+} \mathbf{1}_{\{R^r_{\pi_1}(s^-) \wedge R^r_{\pi_2}(s^-)\leq h\leq R^r_{\pi_1}(s^-)\vee R^r_{\pi_2}(s^-)\}} \overline{\cal{N}}_{\nu,1}(ds,dh)\nonumber\\
         & \qquad +2 \iint_{[0,t]\times \R_+} \mathbf{1}_{\{ B_{\pi_1}(s^-) \wedge B_{\pi_2}(s^-) \leq h\leq B_{\pi_1}(s^-) \vee B_{\pi_2}(s^-)\}} \overline{\cal{N}}_{\lambda,2}(ds,dh)\nonumber\\
         & \qquad \qquad + \iint_{[0,t]\times \R_+} \mathbf{1}_{\{V^r_{\pi_1}(s^-)\wedge V^r_{\pi_2}(s^-) \leq h\leq V^r_{\pi_1}(s^-)\vee V^r_{\pi_2}(s^-)\}}\overline{\cal{N}}_{\nu,2}(ds,dh).
         \end{align}
  where
\begin{align*}
  A_{\pi}(t)= \mathbf{1}_{\{\|Z_{\pi}(t)\|<K\}}\pi(v(t)>0) \text{ and }\qquad
   B_{\pi}(t)= \mathbf{1}_{\{V_{\pi}(t)>0\}}\pi(\|z(t)\|<K).
       \end{align*}
The mean of each term of the right-hand side of the previous equation is bounded as follows. For the first term,
\begin{align*}
  & \E\left( \iint_{[0,t]\times \R_+} \mathbf{1}_{\{A_{\pi_1}(s^-) \wedge  A_{\pi_2}(s^-) \leq  h \leq   A_{\pi_1}(s^-) \vee  A_{\pi_2}(s^-)\}} \overline{\cal{N}}_{\lambda,1}(ds,dh)\right)\\
  & \leq \lambda  \int_0^t |\pi_1(v(s)>0)-\pi_2(v(s)>0)| ds\\
  & = \lambda  \int_0^t |\pi(z,v_1(s)>0)-\pi(z,v_2(s)>0)| ds\\
  & = \lambda  \int_0^t\int_{\omega=(z_1,z_2)\in \cal{D}_T^2} |\mathbf{1}_{\{v_1(s)>0)\}}-\mathbf{1}_{\{v_2(s)>0)\}}|\pi(d\omega) ds\\
    & \leq \lambda  \int_0^t \int_{\omega=(z_1,z_2)\in \cal{D}_T^2} |v_1(s)-v_2(s)|\wedge 1 \quad\pi(d\omega) ds\\
      & \leq \lambda  \int_0^t \rho_s(\pi_1,\pi_2) ds.
  \end{align*}
For the second term of the right-hand side of equation~\eqref{diffZ_pi},
\begin{align*}
 & \E\left( \int_0^t \sum_{l=1}^{+\infty} |\mathbf{1}_{\{l\leq R_{\pi_1}(s^-)\}}-\mathbf{1}_{\{l\leq R_{\pi_2}(s^-)\}}| \cal{N}_{\mu,l}(ds)\right)\\
  &\leq \mu \int_0^t   \E\left(\sum_{l=1}^{+\infty} |\mathbf{1}_{\{l\leq R_{\pi_1}(s)\}}-\mathbf{1}_{\{l\leq R_{\pi_2}(s)\}}|\right) ds\\
 & \leq \mu \int_0^t   \E\left(| R_{\pi_1}(s)-R_{\pi_2}(s)|\right) ds\\
 & \leq \mu \int_0^t   \E\left( \|Z_{\pi_1}-Z_{\pi_2}\|_{\infty,s}\right) ds.
\end{align*}
For the third term of the right-hand side of~\eqref{diffZ_pi},
\begin{align*}
  & \E\left(  \iint_{[0,t]\times \R_+} \mathbf{1}_{\{R^r_{\pi_1}(s^-) \wedge R^r_{\pi_2}(s^-)\leq h\leq R^r_{\pi_1}(s^-)\vee R^r_{\pi_2}(s^-)\}} \overline{\cal{N}}_{\nu,1}(ds,dh)\right) \\
  &  \leq \nu \int_0^t \E(|R^r_{\pi_1}(s)-R^r_{\pi_2}(s)|) ds\\
  &  \leq \nu \int_0^t \E(\|Z_{\pi_1}-Z_{\pi_2}\|_{\infty,s}) ds.
\end{align*}
For the fourth term,
\begin{align*}
&  \iint_{[0,t]\times \R_+} \mathbf{1}_{\{ B_{\pi_1}(s^-) \wedge B_{\pi_2}(s^-) \leq h\leq B_{\pi_1}(s^-) \vee B_{\pi_2}(s^-)\}} \overline{\cal{N}}_{\lambda,2}(ds,dh)\\
&  \leq \lambda \int_0^t \E\left(  |\mathbf{1}_{\{V_{\pi_1}(s)>0\}}\pi_1(\|z(s)\|<K)- \mathbf{1}_{\{V_{\pi_2}(s)>0\}}\pi_2(\|z(s)\|<K)|\right) ds\\
&  \leq \lambda \int_0^t \E\left(  |\mathbf{1}_{\{V_{\pi_1}(s)>0\}}- \mathbf{1}_{\{V_{\pi_2}(s)>0\}}|\right) ds\\
&  \leq \lambda \int_0^t \E\left(  |V_{\pi_1}(s)-V_{\pi_2}(s)|\right) ds\\
&  \leq \lambda \int_0^t \E\left(  \|Z_{\pi_1}-V_{\pi_2}\|_{\infty,s}\right) ds.
\end{align*}
As for the third term, for the fifth term,
\begin{align*}
  & \E\left(  \iint_{[0,t]\times \R_+} \mathbf{1}_{\{V^r_{\pi_1}(s^-) \wedge V^r_{\pi_2}(s^-)\leq h\leq V^r_{\pi_1}(s^-)\vee V^r_{\pi_2}(s^-)\}} \overline{\cal{N}}_{\nu,2}(ds,dh)\right) \\
  &  \leq \nu \int_0^t \E(\|Z_{\pi_1}-Z_{\pi_2}\|_{\infty,s}) ds.
\end{align*}
Thus
\begin{align*}
 &\E(\| Z_{\pi_1}-Z_{\pi_2} \|_{\infty,t})\leq (2\mu+3\nu+2\lambda)\int_0^t  \E(\| Z_{\pi_1}-Z_{\pi_2} \|_{\infty,s})ds+\lambda \int_0^t \rho_s(\pi_1,\pi_2)ds.
\end{align*}
By Gr\"{o}nwall's inequality,
\begin{align*}
  \E(\| Z_{\pi_1}-Z_{\pi_2} \|_{\infty,t})\leq C_t \int_0^t \rho_s(\pi_1,\pi_2)ds
\end{align*}
with $C_t=\lambda \exp((2\mu+3\nu+2\lambda)t)$. For $t\leq T$,
\begin{align}\label{Gronwall}
\rho_t(\Phi(\pi_1),\Phi(\pi_2))  \leq C_T \int_0^t \rho_s(\pi_1,\pi_2)ds.
\end{align}
It leads to the uniqueness of the solution of $\Phi(\pi)=\pi$. Indeed, if $\pi_1$ and $\pi_2$ are fixed points of $\Phi$, then equation~\eqref{Gronwall} is rewritten
\begin{align*}
 \rho_t( \pi_1,\pi_2)\leq C_T \int_0^t \rho_s(\pi_1,\pi_2)ds.
   \end{align*}
By Gr\"{o}nwall's inequality again, for each $t\leq T$, $\rho_t( \pi_1,\pi_2)=0$
thus $\pi_1=\pi_2$. The existence is proved by an iteration argument. Let $\pi_0\in \cal{P}(\cal{D}_T)$ and $\pi_{n+1}=\Phi(\pi_n)$. By equation~\eqref{Gronwall},
\begin{multline*}
  W_T(\pi_{n+1},\pi_n)\leq \rho_T(\pi_{n+1},\pi_n)\\
  \leq C_T^n \rho_T(\pi_{1},\pi_0)\int_{0\leq s_1 \leq s_2 \ldots s_n \leq T}ds_1\ldots ds_n \leq \frac{(C_TT)^n}{n!} \rho_T(\pi_{1},\pi_0).
\end{multline*}
Thus $(\pi_n)$ converges since $(\cal{P}(\cal{D}_T),W_T)$ is complete. Because of equation~\eqref{Gronwall}, $\Phi$ is continuous for the Skorohod topology and its limit is a fixed point of $\Phi$.
  \end{proof}

  \section{Mean-field limit}\label{SectionMean-FieldLimit} 
  Recall that,  by Theorem~\ref{McKean-Vlasov}, the so-called McKean--Vlasov process $(\overline{Z}(t))$ is the unique solution of  the system of equations~\eqref{uniquelimit}. The empirical distribution $\Lambda^N (t)$ of $(Z^N_i(t), 1\leq i \leq N)$ is defined,
  for  $f$ on $\Sigma_K$, by
  \begin{align*}
 \Lambda^N(t)(f)=\frac{1}{N} \sum_{i=1}^N f( Z^N_i(t))= \frac{1}{N} \sum_{i=1}^N f (R^{r,N}_i(t),R^{N}_i(t),V^{N}_i(t),V^{r,N}_i(t)) 
  \end{align*}
  Process $(\Lambda^N(t))$ takes values in
  \begin{align}\label{YN}
    \cal{Y}^N=\{\Lambda \in \cal{P}(\Sigma_K), \Lambda_{(w,x,y,z)}\in \frac{\N}{N}, (w,x,y,z)\in \Sigma_K,\!\!\!\!\! \sum_{(w,x,y,z)\in \Sigma_K}(x+y+z)\Lambda_{(w,x,y,z)}N=M_N\}.
  \end{align}
      \noindent
  As process $Z^N$ is not Markov, process $(\Lambda^N(t))$ is not Markov. The aim of this section is to prove the mean-field convergence for $(Z^N(t))$, i.e. that the sequence of processes $(\Lambda^N(t))$ converges in distribution to $(\overline{Z}(t))$. It means that, for any function $f$ with finite support, the sequence of processes $(\Lambda^N(t)(f))$ converges in distribution to $(\E(f(\overline{Z}(t)))$.

  \begin{theorem}[Mean-field convergence]\label{mean-field}
     The sequence of empirical distribution process $(\Lambda^N(t))$ converges in distribution to a process $(\Lambda(t))\in \cal{D}([0,T],\cal{P}(\Sigma_K))$ defined by, for $f$ with finite support on $\Sigma_K$,
    \begin{align*}
     \Lambda(t)(f)=\E(f(\overline{Z}(t))) 
    \end{align*}
    with $(\overline{Z}(t))$ the unique solution of equation~\eqref{uniquelimit}. Moreover, for any $k\geq 1$ and for $1\leq i_1< \ldots <i_k\leq N$, the sequence of finite marginals $(Z^N_{i_1}(t),\ldots,Z^N_{i_k}(t))$ converges in distribution to $(\overline{Z}_{i_1}(t),\ldots,\overline{Z}_{i_k}(t))$, where $(\overline{Z}_{i_1}(t))$,$\ldots$,  $(\overline{Z}_{i_k}(t))$ are independent random variables with the same distribution as $(\overline{Z}(t))$.
\end{theorem}

  The last property is the  propagation of chaos. The proof is presented in Section~\ref{convergenceempirical}. 

  \subsection{Evolution equations of the empirical measure}
  Let us introduce the following notations. For $z\in \Sigma_K$, $f: \Sigma_K \rightarrow \R_+$ and $(e_i,1\leq i\leq 4)$ the canonical basis of $\R^4$ i.e. $e_1=(1,0,0,0),\ldots,e_4=(0,0,0,1) $,
  \begin{align*}
    \Delta_{i,i+1}(f)(z) &= f(z-e_i+e_{i+1})-f(z),\quad 1\leq i\leq 3,\\
    \Delta_{1}^+(f)(z) &= f(z+e_1)-f(z),\\
    \Delta_{4}^-(f)(z) &= f(z-e_4)-f(z). 
  \end{align*}
 Let $ \Sigma_K=\{(w,x,y,z)\in \N^4, w+x+y+z\leq K\}$. For $f: \Sigma_K\rightarrow \R_+$, straightforwardly by equations~\eqref{state1}-\eqref{state4},
  \begin{align*}
    df(Z_i^N(t))=&\Delta_{1}^+(f)(Z_i^N(t))\sum_{j=1}^N \mathbf{1}_{\{V^N_{j}(t^-)>0,\; S^N_{i}(t^-)<K\}}\; \cal{N}_{\lambda,j}^{U,N}(dt,\{i\})\\
    &+\Delta_{1,2}(f)(Z_i^N(t)) \sum_{j=1}^N\left( \sum_{l=1}^{\infty}\mathbf{1}_{\{l\leq X^N_{j,i}(t^-)\}}\cal{N}_{\nu,j,i,l}(dt)\right)\\
    &+\Delta_{2,3}(f)(Z_i^N(t))\sum_{l=1}^{\infty}\mathbf{1}_{\{l\leq R^N_i(t^-)\}}\cal{N}_{\mu,i,l}(dt)\\
    &+\Delta_{3,4}(f)(Z_i^N(t))\sum_{j=1}^N \mathbf{1}_{\{V^N_i(t^-)>0,\; S^N_j(t^-)<K\}}\; \cal{N}_{\lambda,i}^{U,N}(dt,\{j\})\\
    &+\Delta_{4}^-(f)(Z_i^N(t))\sum_{j=1}^N \left( \sum_{l=1}^{\infty}\mathbf{1}_{\{l\leq X^N_{i,j}(t^-)\}}\cal{N}_{\nu,i,j,l}(dt)\right)
  \end{align*}
  Thus, using the martingale decomposition for Poisson processes,
  \begin{align}\label{f_evolution}
   df(Z_i^N(t))=&\Delta_{1}^+(f)(Z_i^N(t))\mathbf{1}_{\{ S^N_{i}(t)<K\}}\sum_{j=1}^N \mathbf{1}_{\{V^N_{j}(t)>0\}}\frac{\lambda}{N} dt\\
    &+\Delta_{1,2}(f)(Z_i^N(t)) R^{r,N}_{i}(t) \nu dt 
    +\Delta_{2,3}(f)(Z_i^N(t)) R^N_i(t) \mu dt \nonumber\\
    &+\Delta_{3,4}(f)(Z_i^N(t)) \mathbf{1}_{\{V^N_i(t)>0\}} \sum_{j=1}^N \mathbf{1}_{\{ S^N_j(t)<K\}} \; \frac{\lambda}{N} dt\nonumber\\
    &+\Delta_{4}^-(f)(Z_i^N(t)) V^{r,N}_{i}(t) \nu dt +d \cal{M}_{f,i}^N (t) \nonumber
  \end{align}
  where $(\cal{M}_{f,i}^N (t))$ is a martingale which will be detailed in Section~\ref{Smartingale}.
  By summing equation~\eqref{f_evolution} for i from $1$ to $N$ and dividing by $N$,
  \begin{align}\label{empirical_evolution}
    \Lambda^N(t)(f)& = \Lambda^N(0)(f)+\cal{M}_f^N(t)\nonumber\\
    &+\lambda \int_0^t \Lambda^N(s)(\Sigma_K\cap \{y>0\})\Lambda^N(s)(\Delta_{1}^+(f)\mathbf{1}_{\Sigma_{<K}})ds\nonumber \\
    &+  \nu \int_0^t \Lambda^N(s)(\Delta_{1,2}(f)p_1)ds+\mu \int_0^t \Lambda^N(s)(\Delta_{2,3}(f)p_2)ds\\
    &+  \lambda \int_0^t \Lambda^N(s)(\Sigma_{<K})\Lambda^N(s)(\Delta_{3,4}(f)\mathbf{1}_{\{y>0\}})ds+\nu \int_0^t \Lambda^N(s)(\Delta_{4}^-(f)p_4)ds \nonumber
  \end{align}
  where
  \begin{align}
    \cal{M}_f^N(t)&=\frac{1}{N} (\cal{M}_{f,1}^N (t)+\cal{M}_{f,2}^N (t)\ldots+ \cal{M}_{f,N}^N (t)), \label{marte}\\
    \Sigma_{<K}&= \{(w,x,y,z)\in \N^4, w+x+y+z<K\}\label{def<K}
  \end{align}
  and
  \begin{align}
     p_i&:\N^4\rightarrow \N \text{ is the } i-th \text{ projection } (\text{for example } p_1(w,x,y,z)=w).\label{p_i}
  \end{align}

  \subsection{The martingale term}\label{Smartingale}
  The martingale term is  $\cal{M}_f^N(t)$ given by equation~\eqref{marte}  where, for $1\leq i\leq N$,
  \begin{align*}
 d\cal{M}_{f,i}^N (t) =  &\Delta_{1}^+(f)(Z_i^N(t))\sum_{j=1}^N \mathbf{1}_{\{V^N_{j}(t^-)>0,\; S^N_{i}(t^-)<K\}}\; ( \cal{N}_{\lambda,j}^{U,N}(dt,\{i\})-\frac{\lambda}{N}dt )\\
    &+\Delta_{1,2}(f)(Z_i^N(t)) \sum_{j=1}^N \sum_{l=1}^{\infty}\mathbf{1}_{\{l\leq X^N_{j,i}(t^-)\}} ( \cal{N}_{\nu,j,i,l}(dt)-\nu dt )\\
    &+\Delta_{2,3}(f)(Z_i^N(t))\sum_{l=1}^{\infty}\mathbf{1}_{\{l\leq R^N_i(t^-)\}} ( \cal{N}_{\mu,i,l}(dt)-\mu dt )\\
    &+\Delta_{3,4}(f)(Z_i^N(t))\sum_{j=1}^N \mathbf{1}_{\{V^N_i(t^-)>0,\; S^N_j(t^-)<K\}}\; ( \cal{N}_{\lambda,i}^{U,N}(dt,\{j\})-\frac{\lambda}{N} dt )\\
    &+\Delta_{4}^-(f)(Z_i^N(t))\sum_{j=1}^N  \sum_{l=1}^{\infty}\mathbf{1}_{\{l\leq X^N_{i,j}(t^-)\}} ( \cal{N}_{\nu,i,j,l}(dt)-\nu dt ) 
\end{align*}
  whose increasing process is expressed as follows
  \begin{align*}
   \langle  \cal{M}_f^N\rangle (t)=\frac{1}{N^2}(\lambda I_1^N(t)+\mu I_2^N(t)+\nu I_3^N(t))
   \end{align*}
  where, from careful calculations,
  \begin{align*}
    I_1^N(t)= &\sum_{i=1}^N \int_0^t \bigg( \left(\Delta_{1}^+(f)(Z_i^N(s))\right)^2 \mathbf{1}_{\{ S^N_{i}(s)<K\}}\Lambda^N(s)(\Sigma_K\cap \{y>0\})\\
    &  + \left(\Delta_{3,4}(f)(Z_i^N(s))\right)^2 \mathbf{1}_{\{V^N_i(s)>0\}} \Lambda^N(s)(\Sigma_{<K})\\
    &  +\frac{2}{N} \Delta_{1}^+(f)(Z_i^N(s))\; \Delta_{3,4}(f)(Z_i^N(s))\mathbf{1}_{\{V^N_{i}(s)>0,\; S^N_{i}(s)<K\}}\bigg) ds\\
    I_2^N(t)= &\sum_{i=1}^N \int_0^t \left(\Delta_{2,3}(f)(Z_i^N(s))\right)^2 R^N_i(s)ds\\
    I_3^N(t)= &\sum_{i=1}^N \int_0^t \bigg( \left(\Delta_{1,2}(f)(Z_i^N(s))\right)^2 R^{r,N}_i(s)+ \left(\Delta_{4}^-(f)(Z_i^N(s))\right)^2 V^{r,N}_i(s)\\
    &  +\frac{2}{N} \Delta_{1,2}(f)(Z_i^N(s))\; \Delta_{4}^-(f)(Z_i^N(s))X_{i,i}(s)
    \bigg) ds.
  \end{align*}
  The term with $X_{i,i}^N(s)$ comes from the fact that only sequences $(\cal{N}_{\nu,i,j,.},j\not =i)$ and $(\cal{N}_{\nu,j,i,.},j\not =i)$ are independent. See Remark~\ref{rem}.
  It yields then straightforwardly that there exist $C_0,\;C_1>0$ such that, for $1\leq i \leq 3$,
  \begin{align*}
    \| I_i^N\|_{\infty,T}\leq (C_0 N+C_1)T \|f\|^2_{\infty}.
  \end{align*}
  Thus, for 
  \begin{align}\label{martingale}
    \| \langle \cal{M}^N_f\rangle \|_{\infty,T}\leq (\lambda+\mu+\nu)\left( \frac{C_0}{N}+\frac{C_1}{N^2}\right)T\|f\|^2_{\infty}
    .
  \end{align}
  Thus, applying Cauchy-Schwarz then Doob's inequalities,
  \begin{align*}
    \left(\E\left( \sup_{0\leq s\leq T} |\cal{M}^N_f(s)|\right) \right) ^2 & \leq \E\left( \sup_{0\leq s\leq T} |\cal{M}^N_f(s)|^2\right)\\
      & \leq 4 \E((\cal{M}^N_f)^2(T))\\
      & =4 \E(\langle \cal{M}^N_f\rangle(T)),
  \end{align*}
 and  the martingale $(\cal{M}^N_f(t))$ converges in distribution to $0$ when $N$ tends to $\infty$.

 \subsection{Convergence of the empirical measure process}\label{convergenceempirical}
 This section is devoted to the proof of the mean-field convergence theorem (Theorem \ref{mean-field}). The proof is based on tightness and uniqueness    standard arguments using stochastic calculus and martingale theory. To be self-contained, the paper presents the detailed proof, via Propositions \ref{tightness} and \ref{uniqueness}.
 
  \begin{proposition}[Tightness of the empirical measure process]\label{tightness}
    The sequence $(\Lambda^N(t))$ is tight with respect to the convergence in distribution in $\cal{D}(\R_+,\cal{Y}^N )$, $\cal{Y}^N$ defined by equation~\eqref{YN}. Any limiting point $(\Lambda(t))$ is a continuous process with values in
    \begin{align}\label{Y}
      \cal{Y}=\{\Lambda \in \cal{P}(\Sigma_K),\sum_{(w,x,y,z)\in \Sigma_K} (x+y+z)\Lambda_{(w,x,y,z)}=s\},
    \end{align}
    solution of
    \begin{align}\label{ODE}
      \Lambda(t)(f)& = \Lambda(0)(f)+\lambda \int_0^t \Lambda(s)(\Sigma_K\cap \{y>0\})\Lambda(s)(\Delta_{1}^+(f)\mathbf{1}_{\Sigma_{<K}})ds\nonumber \\
    &+  \nu \int_0^t \Lambda(s)(\Delta_{1,2}(f)p_1)ds+\mu \int_0^t \Lambda(s)(\Delta_{2,3}(f)p_2)ds\nonumber \\
    &+  \lambda \int_0^t \Lambda(s)(\Sigma_{<K})\Lambda(s)(\Delta_{3,4}(f)\mathbf{1}_{\{y>0\}})ds+\nu \int_0^t \Lambda(s)(\Delta_{4}^-(f)p_4)ds
    \end{align}
    for any function $f$ on $\Sigma_K=\{(w,x,y,z)\in \N^4, w+x+y+z\leq K\}$ and $\Sigma_{<K}$  and the $p_i$'s defined by equations~\eqref{def<K} and~\eqref{p_i}.
  \end{proposition}
   \begin{proof} This amounts to proving that, for any function  $f$ on $\Sigma_K$, the sequence of processes $(\Lambda^N(t)(f))$ is 
     tight with respect to the topology of the uniform norm on compact sets. For this, using the  the modulus of continuity criterion (see Billinsley \cite{billingsley1999convergence}), it  suffices to prove that, for $T, \;\varepsilon, \;\eta>0$, there exist $\delta_0>0$ and $N_0\in \N$ such that, for all $\delta<\delta_0$ and all $N\geq N_0$,
     \begin{align}\label{modulus}
       \P\left( \underset{|s-t|<\delta}{\underset{0\leq s\leq t\leq T}{\sup}} |\Lambda^N(t)(f))-\Lambda^N(s)(f))|>\eta\right) <\varepsilon.
     \end{align}
     Let $T>0$, $\varepsilon>0$, $\eta>0$, $\delta>0$ and $s,t \in [0,T]$ such that $|s-t|<\delta$ be fixed. For the third term of the right-hand side of equation~\eqref{empirical_evolution}, there exists $C_2>0$ such that
     \begin{align*}
      \left| \int_s^t \Lambda^N(u)(\Sigma_K\cap \{y>0\})\Lambda^N(u)(\Delta_{1}^+(f)\mathbf{1}_{\{\Sigma_{<K}\}})du\right| \leq \delta C_2 \|f\|_{\infty}
     \end{align*}
     and the same holds for the other terms. Thus there exists $C_3>0$ such that
     \begin{align*}
      |\Lambda^N(t)(f))-\Lambda^N(s)(f))| \leq \delta C_3 \|f\|_{\infty}+| \cal{M}^N_f(t)-\cal{M}^N_f(s)|.
     \end{align*}
     Using equation~\eqref{martingale} for the martingale term, it holds that there exists $C_4>0$ such that
     \begin{align*}
     \E\left( \underset{|s-t|<\delta}{\underset{0\leq s\leq t\leq T}{\sup}} |\Lambda^N(t)(f))-\Lambda^N(s)(f))|\right) \leq \delta C_4 \|f\|_{\infty}+2\E\left(\underset{0\leq  t\leq T}{\sup} |\cal{M}_f^N(t)|\right).
     \end{align*}
     Thus, as the martingale term converges in distribution to $0$, there exist $\delta_0>0$ and $N_0\in \N$ such that, for all $\delta<\delta_0$ and all $N\geq N_0$, the left-hand side of the previous equation is less than $\eps$.
     Then, using Markov's inequality, the sequence of processes $(\Lambda^N(t)(f))$ satisfies equation~\eqref{modulus} thus is tight. Therefore, if $\Lambda$ is a limiting point of $\Lambda^N$, using again equation~\eqref{empirical_evolution}, as $(\cal{M}^N_f(t))$ converges in distribution to $0$,
     equation~\eqref{ODE} holds. As function $f$ has finite support, all the terms of the right-hand side are straightforwardly continuous on $t$.
   \end{proof}

   The following proposition  gives the uniqueness of a limiting point of $(\Lambda^N(t))$.
\begin{proposition}[Uniqueness]\label{uniqueness}
    For every probability $\Lambda_0$ on $\Sigma_K$, equation~\eqref{ODE} has at most one solution $(\Lambda(t))$ in $\cal{D}(\R_+, \cal{P}(\Sigma_K))$ with initial condition $\Lambda_0$.
\end{proposition}
\begin{proof}
  Let  $\Lambda^1(t))$ and $\Lambda^2(t))$ in $\cal{D}(\R_+, \cal{P}(\Sigma_K))$ be two solutions of equation~\eqref{ODE} with initial condition $\Lambda_0$. For $f$  function on $\Sigma_K$ and $t\geq 0$, 
  \begin{align*}
   &  \Lambda^1(t)(f)-\Lambda^2(t)(f)\\
    &=  \lambda \int_0^t (\Lambda^1(s)-\Lambda^2(s))(\Sigma_K\cap \{y>0\})\Lambda^1(s)(\Delta_{1}^+(f)\mathbf{1}_{\Sigma_{<K}})ds\nonumber \\
   & +\lambda \int_0^t \Lambda^2(s)(\Sigma_K\cap \{y>0\})(\Lambda^1(s)-\Lambda^2(s))(\Delta_{1}^+(f)\mathbf{1}_{\Sigma_{<K}})ds\nonumber \\
    &+  \nu \int_0^t (\Lambda^1(s)-\Lambda^2(s))(\Delta_{1,2}(f)p_1+\Delta_{4}^-(f)p_4)ds+\mu \int_0^t (\Lambda^1(s)-\Lambda^2(s))(\Delta_{2,3}(f)p_2)ds\nonumber \\
    &+  \lambda \int_0^t (\Lambda^1(s)-\Lambda^2(s))(\Sigma_{<K})\Lambda^1(s)(\Delta_{3,4}(f)\mathbf{1}_{\{y>0\}})ds\\
    &+  \lambda \int_0^t \Lambda^2(s)(\Sigma_{<K})(\Lambda^1(s)-\Lambda^2(s))(\Delta_{3,4}(f)\mathbf{1}_{\{y>0\}})ds.
  \end{align*}
  Recall that, for a signed measure $\pi$  on $\Sigma_K$,
   \begin{align*}
  \|\pi \|_{TV}=\sup\{\left| \pi( f)\right| , f:\Sigma_K\rightarrow \R, \|f\|_{\infty}\leq 1\}.
   \end{align*}
  From the previous equation, it holds that
  \begin{align*}
   \| \Lambda^1(t)-\Lambda^2(t)\|_{TV}\leq (8\lambda+4\nu+2\mu) \int_0^t \| \Lambda^1(s)-\Lambda^2(s)\|_{TV} ds.
  \end{align*}
  Applying Gr\"onwall's lemma,
  \begin{align*}
    \| \Lambda^1(t)-\Lambda^2(t)\|_{TV}=0.
   \end{align*} 
  It ends the proof.
\end{proof}

{\bf Proof of Theorem~\ref{mean-field}}.
  Let $(w,x,y,z)\in \Sigma_K$  and $\Lambda_0=\delta_{(w,x,y,z)}$.  If $(\overline{Z}(t))$ is the
unique solution of equation~\eqref{uniquelimit}  and the measure valued process $(\Lambda(t))$  is defined, for $f$  function  on $\Sigma_K$, by
\begin{align*}
   \Lambda(t)(f)=\E(f(\overline{Z}(t))) 
  \end{align*}
then it is easy to check that $(\Lambda(t))$ is a solution of equation~\eqref{ODE}.
 The convergence
of $(\Lambda^N(t))$ follows from Propositions 1 and 2. 
    See Sznitman \cite[Proposition 2.2]{Sznitman1991topics} for  the propagation of chaos property.

\subsection{Probabilistic interpretation of the asymptotic process}
Note that the Fokker--Planck equation~\eqref{FPequation} is the functional form of the stochastic equation~\eqref{ODE}.  Recall that, in equation~\eqref{FPequation}, equalities in distribution hold, as
  \begin{align*}
 \iint_{[0,t]\times \R_+} \mathbf{1}_{\{0 \leq h \leq \overline{R^r}(s^-)\}}\overline{\cal{N}}_{\nu}(ds,dh)=   \int_0^t \sum_{l=0}^{\infty} \mathbf{1}_{\{l\leq \overline{R^r}(s^-)\}} \cal{N}_{\nu,l}(ds).
  \end{align*}
  Thus the non-homogeneous Markov process $(\overline{Z}(t))$ can be seen as the state process of four queues in tandem, with overall capacity $K$. This means that   $\overline{R^r}(t)$, $ \overline{R}(t)$, $\overline{V}(t)$ and $\overline{V^r}(t)$    are the numbers of customers in respectively
\begin{itemize}
\item[-] the first queue,  an infinite-server queue with service rate $\nu$,
\item[-] the second one, an infinite-server queue with service rate $\mu$,
\item[-] the third one, a one-server queue with variable service rate  $\lambda \P(\overline{S}(t)<K)$ at time $t$
\item[-]  and the last one, an infinite-server queue  with service rate $\nu$,
\end{itemize}
  while the arrival process is an non-homogeneous  Poisson process with intensity $\lambda \P(\overline{V}(t)>0)dt$. 

  \begin{figure}[!t]
\centering
\begin{tikzpicture}[scale=0.8]

\node (arrival) at (-2,0.5) {$\lambda \P(\overline{V}(t)>0) $};
\draw (arrival);
\draw[->] (-2,0) -- (1.2,0);

\draw[-] (1.5,-1.5) -- (1.5,1.5);
\draw[-] (2.5,-1.5) -- (2.5,1.5);
\draw[-] (1.5,1.5) -- (2.5, 1.5);
\draw[-] (1.5,1.2) -- (2.5, 1.2);
\draw[-] (1.5,0.9) -- (2.5, 0.9);
\draw[-] (1.5,0.6) -- (2.5, 0.6);

\node (Rr) at (2,2) {$R^r$};
\draw (Rr);
\node (etaR) at (2,-2) {rate $\eta_1$};
\draw (etaR);


\node (middle1) at (2.8,1.5) {$\nu$};
\draw (middle1);
\draw[->] (2.8,0) -- (4.2,0);

\draw[-] (4.5,-1.5) -- (4.5,1.5);
\draw[-] (5.5,-1.5) -- (5.5,1.5);
\draw[-] (4.5,1.5) -- (5.5, 1.5);
\draw[-] (4.5,1.2) -- (5.5, 1.2);
\draw[-] (4.5,0.9) -- (5.5, 0.9);
\draw[-] (4.5,0.6) -- (5.5, 0.6);

\node (R) at (5,2) {$R$};
\draw (R);
\node (roR) at (5,-2) {rate $\rho_1$};
\draw (roR);


\node (middle2) at (5.8,1.5) {$\mu$};
\draw (middle2);
\draw[->] (5.8,0) -- (7.2,0);

\draw[-] (7.5,.5) -- (9.5,.5);
\draw[-] (7.5,-.5)-- (9.5,-.5);
\draw[-] (9.5,.5) -- (9.5, -.5);
\draw[-] (9,.5) -- (9, -.5);
\draw[-] (8.5,.5) -- (8.5, -.5);

\node (V) at (8.4,1.4) {$V$};
\draw (V);
\node (ro2) at (8.5,-2) {rate $\rho_2$};
\draw (ro2);


\node (middle3) at (10.1,0.8) {\tiny{$\lambda \P(\overline{S}(t)<K)$}};
\draw (middle3);
\draw[->] (9.8,0) -- (11.2,0);

\draw[-] (11.5,-1.5) -- (11.5,1.5);
\draw[-] (12.5,-1.5) -- (12.5,1.5);
\draw[-] (11.5,1.5) -- (12.5, 1.5);
\draw[-] (11.5,1.2) -- (12.5, 1.2);
\draw[-] (11.5,0.9) -- (12.5, 0.9);
\draw[-] (11.5,0.6) -- (12.5, 0.6);

\node (Vr) at (12,2) {$V^r$};
\draw (Vr);
\node (etaR) at (12,-2) {rate $\eta_2$};
\draw (etaR);


\node (out) at (12.8,1.5) {$\nu$};
\draw (out);
\draw[->] (12.8,0) -- (14.2,0);

\end{tikzpicture}
\caption{Dynamics of $(\overline{Z}(t))$ as a tandem of four queues. The {\em vertical} queues are $M/M/\infty$ queues, the {\em horizontal} is a $M/M/1$ queue. The overall capacity is $K$.}
\label{fig:MMqueuesModel3}
  \end{figure}
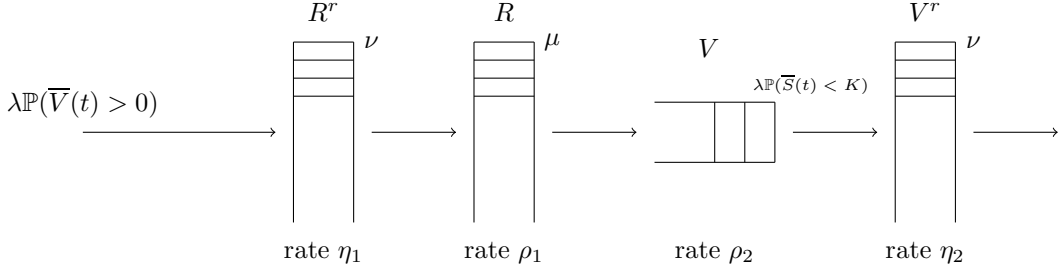
  
  Let  $\eta_1$, $\rho_1$, $\rho_2$ and $\eta_2$ be the arrival-to-service rate ratios for the four queues from the left to the right (as shown on Figure~\ref{fig:MMqueuesModel3}). By definition, one gets
  
  \begin{equation}\label{rhot}
    \begin{cases}
    \eta_1(t) &= \frac{\lambda}{\nu}\P(\overline{V}(t)>0), \quad   \rho_1 (t) = \frac{\lambda}{\mu} \P(\overline{V}(t)>0), \\
    \rho_2 (t) &= \dfrac{\P(\overline{V}(t)>0)}{\P(\overline{S}(t)<K)}, \quad \eta_2(t)  = \frac{\lambda}{\nu}\P(\overline{V}(t)>0).
    \end{cases}
\end{equation}

\section{Equilibrium of the asymptotic process}\label{SectionEquilibrium}

The quadruplet of the numbers of customers in such four queues in tandem with fixed arrival-to-service rate ratios $\eta_1$, $\rho_1$, $\rho_2$ and $\eta_2$ and finite overall capacity $K$ is an ergodic Markov process as an irreducible Markov process  on a finite state space. Moreover, the unique invariant probability  measure has a well-known product form 
given by 

\begin{align} \label{MeasureModel3}
\pi_{j,k,l,m}(\rho) = \frac{1}{Z(\rho)} \frac{\eta_1^j \rho_1^k}{j!k!} \rho_2^l \frac{\eta_2^m}{m!} 
\end{align}
where, to  shorten the notations,  $(\eta_1, \rho_1, \rho_2, \eta_2)$ is denoted by $\rho$ and  the normalizing constant is
\begin{align}
  Z( \rho)=\sum_{j+k+l+m\leq K}  \frac{\eta_1^j \rho_1^k}{j!k!} \rho_2^l \frac{\eta_2^m}{m!} .
 \end{align}

\noindent
If the process $(\overline{Z}(t))$ of the number of customers in  the four queues in tandem is at equilibrium then, denoting by $L_{\rho(t)}$ its  generator,  it holds that
\begin{align*}
  0=\pi(\rho(t)) L_{\rho(t)}
\end{align*}
where $\rho(t)=(\eta_1, \rho_1, \rho_2, \eta_2)(t)$ is given by equation~\eqref{rhot}. It means  that the equilibrium point is the probability measure $\pi(\rho)$ on $\cal{Y}$ defined by~\eqref{Y} where $\rho=(\eta_1, \rho_1, \rho_2, \eta_2)$ satisfies
  \begin{align}
     \eta_1 &= \frac{\lambda}{\nu}(1-\pi_{0V}(\rho)), \label{eta1}\\
    \rho_1  &= \frac{\lambda}{\mu} (1-\pi_{0V}(\rho)) , \label{rho1}\\
    \rho_2  &= \frac{1-\pi_{0V}(\rho)}{1-\pi_{S}(\rho)}, \label{rho2}\\
\eta_2  &= \frac{\lambda}{\nu}(1-\pi_{0V}(\rho)) \label{eta2}
  \end{align}
with  $\pi_{0V}(\rho)=\sum_{j+l+m\leq K} \pi_{j,0,l,m}(\rho)$ and
  $\pi_{S}(\rho)=\sum_{j+k+l+m = K} \pi_{j,k,l,m}(\rho)$. Viewing the tandem of four queues  as a station, $\pi_{0V}(\rho)$ is the probability that there is no car available  and $\pi_S(\rho)$ the probability that the station is saturated.
  
  \noindent
Because $\pi(\rho)$ has support on  $\cal{Y}$, it holds that
\begin{align}\label{s}
  s=\sum_{j+k+l+m\leq K} (k+l+m) \pi_{j,k,l,m}(\rho).
\end{align}

\begin{theorem}[Uniqueness and characterization of the equilibrium point]\label{Equilibrium}
  If $\nu$ is enough large then there exists a unique equilibrium point for the Fokker--Planck equation~\eqref{ODE} which is  $\pi(\rho)$ defined by equation~\eqref{MeasureModel3}  where $\rho= (\frac{\mu}{\nu} \rho_1,\rho_1, \rho_2,\frac{\mu}{\nu} \rho_1)$
  and $(\rho_1,\rho_2)$ 
  is the unique  solution of
   \begin{align}
\rho_1 &= \frac{\lambda}{\mu} (1-\pi_{0V}(\rho)) \label{FPE1},\\
s&= \sum_{j+k+l+m\leq K} (k+l+m) \pi_{j,k,l,m}(\rho). \label{FPE2}
   \end{align}
   \end{theorem}


 \begin{remark}
The uniqueness of the equilibrium point is the main issue. Moreover, its characterization given by Theorem 3 is a major contribution which allows to derive quantitative results. Proving the uniqueness of the equilibrium point in dimension $2$ is quite rare in the literature. We can cite two papers. For the celebrated model by Gibbens, Hunt and Kelly \cite{gibbens1990bistability}, simulations highlight in this paper a range of parameters with two stable equilibrium points, called metastability phenomena. This fact was proved some $30$ years later  in \cite{martirosyan2020theequilibrium}. In \cite{baccelli2022migration}, the same question is solved for a migration-contagion model using a convexity argument. The arguments used in our paper are totally different.
  \end{remark}

\begin{remark}
  The assumption that $\nu$ is enough large is a technical assumption for the proof. It seems that the result is true for all $\nu>0$ but the proof is for the moment out of reach. 
  \end{remark}

Let us prove Theorem~\ref{Equilibrium} in five steps. Steps 2, 3 and 4  are devoted to the special case where $\nu$ tends to infinity. This case is called the simple reservation case. Indeed, when $\nu$ gets large, it turns out that the model corresponds to the case where the car is not reserved in advance, and the parking space is just reserved when the user takes the car. This model is studied in \cite{bourdais2020mean}.  Step 1 is here to present the  framework for the simple reservation case.   Steps 2, 3 and 4 exhibit the proof of  \cite[Theorem 2]{bourdais2020mean} omitted in~\cite{bourdais2020mean} for the simple reservation model.

{\footnotesize \bf STEP 1: THE SIMPLE RESERVATION CASE.}
For the model with simple reservation, the problem of existence and uniqueness of an equilibrium point amounts to finding $(\rho_1,\rho_2)$ such that
\begin{align}
         \rho_1  &= \frac{\lambda}{\mu} (1-\pi_{.,0}(\rho_1,\rho_2)) , \label{rho1_SR}\\
         \rho_2  &= \frac{1-\pi_{.,0}(\rho_1,\rho_2)}{1-\pi_{S}(\rho_1,\rho_2)}, \label{rho2_SR}\\
          s &= \sum_{i+j\leq K} (i+j) \pi_{i,j}(\rho_1, \rho_2).\label{s_SR}
\end{align}
where the invariant probability measure is now
\begin{align*}
  \pi_{i,j}(\rho_1,\rho_2)=\frac{1}{Z(\rho_1,\rho_2)}\frac{\rho_1^i}{i!}\rho_2^j
\end{align*}
with $Z(\rho_1,\rho_2)=
\sum_{i+j\leq K} \pi_{i,j}(\rho_1,\rho_2)$, 
$\pi_{.,0}=\sum_{i=0}^K  \pi_{i,0}$ and
  $\pi_{S}=\sum_{i+j = K} \pi_{i,j}$.


{\footnotesize  \bf STEP 2: EQUATIONS~\eqref{rho1_SR}-\eqref{s_SR} AMOUNT TO EQUATIONS~\eqref{rho1_SR} and~\eqref{s_SR}.}
Note that 
any $\rho$ solution of~\eqref{rho1_SR} and~\eqref{rho2_SR} lies in  $\Gamma= [0, \lambda/\mu \mathclose{[}\times \R^+$. For any $\rho \in \Gamma$, $\rho$ is solution of equation~\eqref{rho2_SR} because
       \begin{align*}
\rho_2 (1-\pi_S(\rho_1,\rho_2)) 	= \rho_2 \frac{1}{Z(\rho_1,\rho_2)}\sum_{i+j < K} \frac{\rho_1^i}{i!}\rho_2^j 
= \frac{1}{Z(\rho_1,\rho_2)}\sum_{j>0, i+j \leq K} \frac{\rho_1^i}{i!}\rho_2^j 
							&= 1-\pi_{.,0}(\rho_1,\rho_2).
\end{align*}

{\footnotesize  \bf STEP 3:  DIFFEOMORPHISM FROM EQUATION~\eqref{rho1_SR}. }
\begin{lemma}[Diffeomorphism]
\label{Theorem:diffeomorphism}
There exists a strictly increasing diffeomorphism $\phi: \mathopen{]}0, \lambda/\mu \mathclose{[} \rightarrow ]0, \infty \mathclose{[}$, and $\psi=\phi^{-1}$, such that  $(\rho_1, \rho_2)$ is a solution of equation~\eqref{rho1_SR} if and only if $\rho_2 = \phi(\rho_1)$.
\end{lemma}
\begin{proof}


  First, $(\rho_1, \rho_2) \in \Gamma$  is solution of equation~\eqref{rho1_SR} if and only if $f(\rho_1, \rho_2) = 0$ where $f$ is the  $\mathcal{C}^{\infty}$ function defined by
\begin{align}\label{f}
f(\rho_1, \rho_2) = \left(\frac{\lambda}{\mu}-\rho_1 \right)Z(\rho_1, \rho_2) - \frac{\lambda}{\mu}\sum\limits_{i=0}^K \frac{\rho_1^i}{i!}.
\end{align}
 
To prove the existence of a  strictly increasing diffeomorphism $\phi$ which maps $\rho_1$ to $\rho_2$, we  apply the global inverse function theorem to  auxiliary function $h$ defined on $\Gamma$ by
\begin{align*}
  h(\rho_1,\rho_2)=(\rho_1,f(\rho_1,\rho_2)).
  \end{align*}
Indeed, $h$ is injective if and only if, for any $\rho_1 \in ]0, \lambda/\mu[ $ and $\rho_2,\, \rho'_2 \in ]0, +\infty[$,  $f(\rho_1,\rho_2)=f(\rho_1,\rho_2')$ implies that $\rho_2=\rho_2'$. As $f$ is $C^1$ on $]0, \lambda/\mu[  \times ]0, +\infty[$, it is sufficient to prove that $\partial f/\partial \rho_2 \not = 0 $ on $]0, \lambda/\mu[  \times ]0, +\infty[$  to obtain that $\rho_2\mapsto f(\rho_1,\rho_2)$ is strictly monotone thus injective. 

Note that it is easy to see that $\partial f/\partial \rho_2 $ is positive on $]0, \lambda/\mu[  \times ]0, +\infty[$. Indeed, since $\rho_1 < \lambda/\mu$ and $\rho_2\mapsto Z(\rho_1, \rho_2)$ is non decreasing,
$$\frac{\partial f}{\partial \rho_2}(\rho_1, \rho_2) = \left(\frac{\lambda}{\mu}-\rho_1 \right) \frac{\partial Z}{\partial \rho_2}(\rho_1, \rho_2)  >0  \text{.}$$
As a consequence, $h$ is injective and $C^1$ on $]0, \lambda/\mu[  \times ]0, +\infty[$.  By global inversion function theorem, $h^{-1}$ is $C^1$ on $h \big(]0, \lambda/\mu[  \times ]0, +\infty[\big)$ and it exists $\phi$  $C^1$  defined on $]0,\lambda/\mu[$ by
 \begin{align*}
  h^{-1}(\rho_1,0)=(\rho_1,\phi(\rho_1)).
\end{align*}  
Thus, to prove that  diffeomorphism $\phi$ is strictly increasing on $]0, \lambda/\mu[$, it amounts to showing that, for all $\rho_1 \in ]0, \lambda/\mu[$, one gets
 \begin{align*}
   \frac{\partial f}{\partial \rho_1}(\rho_1, \phi(\rho_1)) < 0,
    \end{align*}
 or equivalently that, for all $(\rho_1, \rho_2) \in ]0, \lambda/\mu[  \times ]0, +\infty[$ such that $f(\rho_1, \rho_2)=0$,
    \begin{align*}
      \frac{\partial f}{\partial \rho_1}(\rho_1, \rho_2) < 0.
       \end{align*}

First, from equation~\eqref{f}, one obtains the first partial derivative of function $f$,
\begin{align}\label{df/dr_R}
\frac{\partial f}{\partial \rho_1}(\rho_1, \rho_2) = -Z(\rho_1, \rho_2) 
+ \left(\frac{\lambda}{\mu}-\rho_1 \right) \sum_{i+j \leq K-1} \frac{\rho_1^i}{i!}  \rho_2^j-\frac{\lambda}{\mu} \sum\limits_{i=0}^{K-1}\frac{\rho_1^i}{i!}
\end{align}
thus, subtracting equation~\eqref{df/dr_R} to equation~\eqref{f} yields
\begin{align}\label{diff}
 f(\rho_1, \rho_2)-\frac{\partial f}{\partial \rho_1}(\rho_1, \rho_2) = Z(\rho_1, \rho_2) 
+ \left(\frac{\lambda}{\mu}-\rho_1 \right) \sum_{j=0}^K \frac{\rho_2^j \rho_1^{K-j}}{(K-j)!} - \frac{\lambda}{\mu} \frac{\rho_1^K}{K!}.
\end{align}
By equation~\eqref{f}, $f(\rho_1, \rho_2)=0$ can be rewritten
\begin{align*}
\frac{\lambda}{\mu}-\rho_1=\frac{\lambda}{\mu Z(\rho_1, \rho_2)}\sum\limits_{i=0}^K \frac{\rho_1^i}{i!}.
\end{align*}
Thus, the second term of the right-hand side of equation~\eqref{diff} is 
\begin{multline}\label{Majoration}
  \left(\frac{\lambda}{\mu}-\rho_1 \right) \sum_{j=0}^K \frac{\rho_2^j \rho_1^{K-j}}{(K-j)!}= \frac{\lambda}{\mu} \frac{1}{Z(\rho_1, \rho_2)} \sum\limits_{i=0}^K \frac{\rho_1^i}{i!} \sum\limits_{j=0}^K \frac{\rho_2^j \rho_1^{K-j}}{(K-j)!}   \\
  = \frac{\lambda}{\mu} \rho_1^K \frac{1}{Z(\rho_1, \rho_2)}  \sum\limits_{  i,j=0}^K \frac{\rho_1^{i-j}\rho_2^j}{i!(K-j)!}\geq \frac{\lambda}{\mu} \rho_1^K \frac{1}{Z(\rho_1, \rho_2)}  \sum\limits_{0 \leq j \leq i \leq K} \frac{\rho_1^{i-j}\rho_2^j}{i!(K-j)!}
  \end{multline}
using that all  the terms are positive.
For the sum of the right-hand side of~\eqref{Majoration}, it holds that
\begin{align*}
\sum\limits_{0 \leq j \leq i \leq K} \frac{\rho_1^{i-j}\rho_2^j}{i!(K-j)!} &= \frac{1}{K!} \sum\limits_{j+k\leq K} \frac{K!\rho_1^{k}\rho_2^j}{(j+k)!(K-j)!}  
  \geq \frac{Z(\rho_1, \rho_2)}{K!}
\end{align*}
using that, for any $j,k \in \N, j+k\leq K$, 
\begin{align*}
\frac{K!}{(j+k)!(K-j)!} = \frac{1}{k!} \frac{K!}{(K-j)!}\frac{k!}{(j+k)!}= \frac{1}{k!}\prod_{i=0}^{j-1}\frac{K-i}{j+k-i} \geq \frac{1}{k!}.
\end{align*}
In conclusion,  equation~\eqref{Majoration} gives that
\begin{equation}
\label{Majoration2}
\left(\frac{\lambda}{\mu}-\rho_1 \right) \sum\limits_{k=0}^K \frac{\rho_2^K \rho_1^{K-k}}{(K-k)!} \geq  \frac{\lambda}{\mu} \frac{\rho_1^K}{K!} \text{.}
\end{equation}

Plugging equation~\eqref{Majoration2} in  equation~\eqref{diff} and using that $Z(\rho_1, \rho_2)>0$,  it turns out  that 
$$f(\rho_1, \rho_2)- \frac{\partial f}{\partial \rho_1}(\rho_1, \rho_2) > 0.$$ 
Therefore, as $f(\rho_1, \rho_2)=0$,
$$\frac{\partial f}{\partial \rho_1} (\rho_1, \rho_2)<0  $$
for all $(\rho_1, \rho_2) \in ]0, \lambda/\mu[  \times ]0, +\infty[$ such that $f(\rho_1, \rho_2)=0$.

\end{proof}

Lemma~\ref{Theorem:diffeomorphism}  means that the solutions of equation~\eqref{rho1_SR} can be expressed with only one parameter. Note that it will  be useful for the study of the equilibrium, especially  in calculations to obtain asymptotics. This concludes the third step of the proof.









{\footnotesize  \bf STEP 4: A MONOTONICITY ARGUMENT TO CONCLUDE SIMPLE RESERVATION CASE.}
After equations~\eqref{rho1_SR} and~\eqref{rho2_SR}, let us focus now on  equation~\eqref{s_SR}. The idea is to prove that the mean number in a tandem of two queues with total capacity $K$  is a strictly increasing function of both arrival-to-service rates. It generalises the monotonicity argument in the similar proof in \cite[Section 3.1]{fricker2016incentives}. Then using Lemma~\ref{Theorem:diffeomorphism}, it  concludes to the existence and uniqueness of the equilibrium point.
\begin{lemma}[Monotonicity]
\label{sCroissante}
The average number $\mathbb{E}(R+V)$ of vehicles and reserved spaces per station,  
 where $(R,V)$ is a random variable with distribution $\pi(\rho_2 ,\rho_1)$,
is a strictly increasing function of both $\rho_2$ and $\rho_1$. 
\end{lemma}
\begin{proof}
  Let $ (\rho_1,\rho_2)\mapsto \E(R+V)$ be denoted by $g_K$. It is sufficient to prove  that, for all $(\rho_1,\rho_2) \in \Gamma$, 
$$\frac{\partial g_K}{\partial \rho_1}(\rho_1,\rho_2) >0 \text{ and }  \frac{\partial g_K}{\partial \rho_2}(\rho_1,\rho_2) >0$$
 by induction on $K$. 

By a change of indexes, $g_K$ can be rewritten
$$g_K  = \frac{\sum\limits_{i+j \leq K} (i+j)\frac{\rho_1^i}{i!}\rho_2^j }{\sum\limits_{i+j \leq K}\frac{\rho_1^i}{i!}\rho_2^j}   =  \frac{\sum\limits_{k=0}^K k p_k}{\sum\limits_{k=0}^K p_k} $$
where, by definition,
$$p_k = \sum\limits_{j=0}^k \rho_2^j \frac{\rho_1^{k-j}}{(k-j)!} \text{.}$$
Define also, for $(k,l)\in \N^2$,
$$r_{l,k} = \frac{p_l}{p_k}.$$

Let $k>0$ be fixed. We first show that $r_{k,k-1}$ is an increasing function of both $\rho_1$ and $\rho_2$. Indeed

$$r_{k,k-1} = \frac{p_k}{p_{k-1}} = \frac{ \rho_2 p_{k-1} + \rho_1^{k}/k! }{p_{k-1}} = \rho_2 + \frac{\rho_1^{k}}{k!}\frac{1}{p_{k-1}} \text{,}$$

thus
\begin{align}\label{dRhoV}
\frac{\partial r_{k,k-1}}{\partial \rho_2} = 1 - \frac{\rho_1^{k}}{k!} \frac{\partial p_{k-1}}{\partial \rho_2} \frac{1}{p_{k-1}^2}.
\end{align}

But, by definition of $p_k$,
\begin{align}\label{arche}
\rho_1^{k} \frac{\partial p_{k-1}}{\partial \rho_2} = \sum\limits_{j=1}^{k-1} j \rho_2^{j-1}\frac{\rho_1^{2k-j-1}}{(k-j-1)!} = \sum\limits_{j=0}^{k-2} (j+1) \rho_2^{j}\frac{\rho_1^{2k-j-2}}{(k-j-2)!}.
\end{align}

And, using that all terms of the sum in the following equation are positive,
\begin{align}\label{biche}
k!p_{k-1}^2 &= k! \sum\limits_{1 \leq u,v \leq k} \rho_2^{u+v-2} \frac{\rho_1^{2k-u-v}}{(k-u)!(k-v)!} \nonumber\\
	&> \sum\limits_{i=0}^{k-2} \rho_2^{i}\rho_1^{2k-i-2} \sum\limits_{j=1}^{i+1} \frac{k!}{(k-2-i+j)!(k-j)!} \text{.} 
\end{align}

For all $j,\;0\leq j  \leq i+1$, while
 $k-2-i+j < k$, it holds that

$$\frac{k!}{(k-2-i+j)!(k-j)!} > \frac{1}{(k-i-2)!}$$
and then 
$$\sum\limits_{j=1}^{i+1} \frac{k!}{(k-2-i+j)!(k-j)!} > \frac{i+1}{(k-i-2)!} \text{.}$$
Plugging in equation~\eqref{biche} and comparing with  equation~\eqref{arche}, it gives
$$k!p_{k-1}^2 > \rho_1^{k} \frac{\partial p_{k-1}}{\partial \rho_2}.$$ 
Therefore, using ~\eqref{dRhoV}, it allows to conclude that $$\frac{\partial r_{k,k-1}}{\partial \rho_2} >0 \text{.}$$

Moreover,
\begin{align}\label{chat}
\frac{\partial r_{k,k-1}}{\partial \rho_1} = \frac{\rho_1^{k-1}}{(k-1)!p_{k-1}} \left( 1 - \frac{\rho_1}{kp_{k-1}}\frac{\partial p_{k-1}}{\partial \rho_1} \right)>0 \text{.}
\end{align}
because it is easily checked that
$\partial p_{k-1}/\partial \rho_1= p_{k-2}$ and then
$kp_{k-1} > \rho_1 p_{k-2}$.

Therefore, if $l>k$, $r_{l,k} = \prod\limits_{i=k+1}^l r_{i,i-1}$ is an increasing function of $\rho_2$ and $\rho_1$. This gives that $u_K$ defined by 

$$u_K=\frac{p_K}{\sum\limits_{k=0}^K p_k}=\frac{1}{\sum\limits_{k=0}^K r_{k,K}}$$
is non decreasing in $x \in \{ \rho_1,\rho_2 \}$, because $r_{k,K}=1/r_{K,k}$ is non increasing in $x$.

Note that $g_0$ is constant with $x$ and that
$$g_K = (1-u_K) g_{K-1} + K u_K,$$
which yields that
$$\frac{\partial g_K}{\partial x} = (K-g_{K-1}) \frac{\partial u_K}{\partial x} + (1-u_K) \frac{\partial g_{K-1}}{\partial x} \text{.}$$

Since $K-g_{K-1} > 0$, $u_K <1$ and $\partial u_K/\partial x >0$, by induction we can conclude that $\partial g_K/\partial x>0$ for all $K \geq 1$. It ends the proof.
\end{proof}

By step 2, it remains to prove that, for any $s>0$, there exists  a unique  $(\rho_1,  \rho_2)$ solution of both equations~\eqref{rho1_SR} and ~\eqref{s_SR}.
By Lemma 1, equation~\eqref{rho1_SR} can be rewritten $\rho_1=\psi_{\rho_2}(\rho_2)$ with $\rho_2 \geq 0$. Then equation~\eqref{s_SR} becomes
\begin{align}\label{s_K}
  s=g_K( \psi_{\rho_2}(\rho_2),\rho_2).
  \end{align}
Let us denote  the right-hand side of equation~\eqref{s_K} by function $s_K$ defined on $[0,+\infty[$   which maps
$\rho_2$ to $g_K(\psi_{\rho_2}(\rho_2),\rho_2)=\E(R+V)$. It is sufficient to prove that  $s_K$  is  strictly increasing.
It is true, using that
$$s_K' =  \frac{\partial g_K}{\partial \rho_2} + \frac{\partial g_K}{\partial \rho_1} \psi'_{\rho_2} $$
as equation~\eqref{s_K} holds with $ \psi_{\rho_2}$ at $\rho_2$ but also in a neighborhood of $\rho_2$ and using both Lemmas~\ref{Theorem:diffeomorphism} and ~\ref{sCroissante}. To conclude the proof, one can easily check that $s_K$ covers the whole interval $[0,+\infty[$. Indeed if $\rho_2 = 0$, the only $\rho_1$ solution of $f(\rho_1,\rho_2)=0$ is $0$ and $E(R+V)=0$. Similarly, when $\rho_2$ tends to infinity, $\rho_1$ has to tend to $\lambda/\mu$ to keep $f(\rho_1, \rho_2) = 0$ and $E(R+V)$ tends to $+\infty$. It ends the proof for the single reservation case.





{\footnotesize  \bf STEP 5: THE DOUBLE RESERVATION CASE.}
Straightforwardlly  equations~\eqref{eta1} and \eqref{eta2} leads  to
\begin{align*}
   \eta_1 = \eta_2 =\frac{\mu}{\nu} \rho_1. 
 \end{align*}
Then the main argument is that equations~\eqref{rho1} and  ~\eqref{rho2}  can be rewritten as equations~\eqref{rho1_SR} and ~\eqref{rho2_SR} where $\rho_1$ is replaced by  $\tilde{\rho_1}=(1+2\mu/\nu)\rho_1$ and $\lambda/\mu$ by $a=\lambda/\mu(1+2\mu/\nu)$. Indeed, by straightforward algebra, it holds that
$\pi_{0V}(\rho)=\pi_{.,0}(\tilde{\rho_1}, \rho_2)$ and $\pi_{S}(\rho)=\pi_{S}(\tilde{\rho_1},\rho_2)$.

Unfortunately,  ~\eqref{s} is not rewritten as ~\eqref{s_SR}  with the previous change of variables  $\tilde{\rho_1}=(1+2\mu/\nu)\rho_1$ but, with careful calculations, as 
\begin{align}
  s &= \sum_{i+j\leq K} \left(\frac{1+\mu/\nu}{1+2\mu/\nu}i+j\right) \pi_{i,j}(\tilde{\rho_1}, \rho_2).\label{s_mod}
  \end{align}
 To complete the  proof for any $\nu>0$, it  amounts to proving  that, for $a>0$ and $K\in \N$, 
\begin{align*}
 S(x,y)= \frac{1}{Z(x,y)} \sum_{i+j\leq K}(i+2j)\frac{x^i}{i!}y^j
\end{align*}
is a strictly increasing function of $x\in [0,a[$ under $f(x,y)=0$,
where
\begin{align*}
 f(x,y)=(a-x)Z(x,y)-a\sum_{i=0}^K \frac{x^i}{i!}, \quad Z(x,y)=\sum_{i+j\leq K}\frac{x^i}{i!}y^j.
\end{align*}

Indeed, the ratio in the right-hand side of equation~\eqref{s_mod}
is a non increasing function from $(0,+\infty)$ to $(1/2,1)$. By linearity of differentiability and multiplication by a scalar, it  suffices to prove that the right-hand side of equation~\eqref{s_mod} is non-increasing as a function of $\tilde{\rho_1}$, for $\rho_2=\phi(\rho_1)$ defined in Lemma~\ref{Theorem:diffeomorphism} for the two values $1/2$ and $1$ of this ratio. For value $1$, it is exactly Lemma~\ref{sCroissante}. It remains to prove it for value $1/2$, which has been asserted previously.

Nevertheless,  due to simple reservation case (steps 2 to 4),  by continuity, the proof is complete for $\nu$ enough large.


\section*{Acknowledgement}
The authors would like to thank C\'edric Bourdais for his contribution to the analysis of this model during his internship at INRIA.


\end{document}